\documentclass[11pt]{amsart}
\usepackage[numbers, square]{natbib}
\usepackage{amssymb}
\usepackage{amsmath}
\usepackage{amsfonts}
\usepackage{geometry}
\usepackage{graphicx}
\usepackage{amsthm}
\usepackage{hyperref}
\usepackage[dvipsnames]{xcolor}

\providecommand{\U}[1]{\protect\rule{.1in}{.1in}}
\theoremstyle{plain}

\newtheorem{corollary}{Corollary}

\newtheorem{lemma}{Lemma}

\newtheorem{remark}{Remark}

\newtheorem{theorem}{Theorem}
\numberwithin{equation}{section}

\newcommand{\disp}{\displaystyle}

\DeclareMathOperator{\osc}{osc}

\newcommand{\eps}{\varepsilon}
\newcommand{\vp}{\varphi}

\newcommand{\al}{\alpha}
\newcommand{\be}{\beta}

\newcommand{\te}{\theta}
\newcommand{\la}{\lambda}

\newcommand{\om}{\omega}
\newcommand{\Om}{\Omega}
\newcommand{\si}{\sigma}

\newcommand{\iny}{\infty}
\newcommand{\del}{ \partial}
\newcommand{\su}{\subset}
\newcommand{\LP}{\Delta}
\newcommand{\gr}{\nabla}

\newcommand{\norm}[1]{\left\vert\left\vert #1\right\vert\right\vert}

\newcommand{\abs}[1]{\left\vert#1\right\vert}
\newcommand{\set}[1]{\left\{#1\right\}}
\newcommand{\brac}[1]{\left[#1\right]}
\newcommand{\pr}[1]{\left( #1 \right) }
\newcommand{\pb}[1]{\left( #1 \right] }
\newcommand{\bp}[1]{\left[ #1 \right) }

\newcommand{\N}{\ensuremath{\mathbb{N}}}

\newcommand{\R}{\ensuremath{\mathbb{R}}}

\newcommand{\C}{\ensuremath{\mathbb{C}}}

\begin{document}
\title[Quantitative uniqueness ]
{Quantitative uniqueness of solutions to second order elliptic equations with singular lower order terms}
\author{ Blair Davey and Jiuyi Zhu}
\address{
Department of Mathematics\\
The City College of New York\\
New York, NY 10031, USA\\
Email: bdavey@ccny.cuny.edu }
\address{
Department of Mathematics\\
Louisiana State University\\
Baton Rouge, LA 70803, USA\\
Email:  zhu@math.lsu.edu }
\thanks{\noindent{Davey is supported in part by the Simons Foundation Grant number 430198.
Zhu is supported in part by \indent NSF grant DMS-1656845}}
\date{}

\begin{abstract}
In this article, we study some quantitative unique continuation properties of solutions to second order elliptic equations with singular lower order terms.
First, we quantify the strong unique continuation property by estimating the maximal vanishing order of solutions.
That is, when $u$ is a non-trivial solution to $\LP u + W \cdot \gr u + V u = 0$ in some open, connected subset of $\R^n$, where $n \ge 3$, we characterize the vanishing order of solutions in terms of the norms of $V$ and $W$ in their respective Lebesgue spaces.
Then, using these maximal order of vanishing estimates, we establish quantitative unique continuation at infinity results for solutions to $\LP u + W \cdot \gr u + V u = 0$ in $\R^n$.
The main tools in our work are new versions of $L^p\to L^q$ Carleman estimates for a
range of $p$- and $q$-values.
\end{abstract}

\maketitle
\section{Introduction}

In this paper, we investigate some quantitative unique continuation properties, or simply the quantitative uniqueness, of solutions to second order elliptic equations with singular lower order terms.
A partial differential operator $P$ defined in $\Om \su \R^n$ is said to have the {\em strong unique continuation property} in the function space $S$ if whenever $u \in S$ is a solution to $P u = 0$ in $\Om$, and $u$ vanishes to infinite order at some point $x_0 \in \Om$, then, necessarily, $u \equiv 0$ throughout $\Om$.
If $P$ has the strong unique continuation property, then it is interesting to determine the fastest rate at which a solution can vanish without being trivial.
We call this rate the {\em maximal order of vanishing}.

We assume throughout that $n \ge 3$.
In a forthcoming paper, we will consider $n=2$ dimensions.
We use the notation $B_{r}\pr{x_0} \su \R^n$ to denote the ball of radius $r$ centered at $x_0$. When the center is understood from the context, we simply write $B_{r}$.

Suppose that for some $K, M \ge 1$, $\norm{W}_{L^\iny\pr{B_{10}}} \le K$ and $\norm{V}_{L^\iny\pr{B_{10}}} \le M$.
If $u : B_{10} \to \C$ is a solution to
\begin{equation}
\LP u + W \cdot \gr u + V u = 0
\label{ePDE}
\end{equation}
in $B_{10}$ with $\norm{u}_{L^\iny\pr{B_1}} \ge 1$ and $\norm{u}_{L^\iny\pr{B_{10}}} \le \hat C$, then a quantitative form of strong unique continuation asserts that
\begin{equation}
\norm{u}_{L^\iny\pr{B_r}} \ge c r^{C\pr{K^2 + M^{2/3}}} \quad \text{
as } r \to 0, \label{OofV}
\end{equation}
which implies that the maximal order of vanishing for $u$ at origin
is less than $C\pr{K^2 + M^{2/3}}$. When $W \equiv 0$ (taking $K =
0$), this maximal order of vanishing estimate was proved by Bourgain
and Kenig in \cite{BK05}. They used this result to establish
estimates at infinity that were relevant to their work on Anderson
localization. Meshkov's examples in \cite{Mes92} imply that the
power of $2/3$ is optimal for complex-valued functions. In
\cite{Dav14}, the first author generalized the work of Bourgain and
Kenig in the presence of a first order term, $W$, by proving an
order of vanishing estimate as in \eqref{OofV} and a quantitative
unique continuation at infinity theorem. The latter theorem takes
the following form: Assume that $\norm{W}_{L^\iny\pr{\R^n}} \le A_1$
and $\norm{V}_{L^\iny\pr{\R^n}} \le A_0$. If $u : \R^n \to \C$ is a
solution to \eqref{ePDE} in $\R^n$ with $\norm{u}_{L^\iny\pr{\R^n}}
\le C_0$ and $\abs{u\pr{0}} \ge 1$, then for all $R$ sufficiently
large,
\begin{equation}
\mathcal{M}\pr{R} := \inf_{\abs{x_0} = R} \norm{u}_{L^\iny\pr{B_1\pr{x_0}}} \ge \exp\pr{-C R^2 \log R}.
\label{UCiny}
\end{equation}
By adapting the constructions of Meshkov from \cite{Mes92}, she showed that the power of $2$ is best possible in the complex-valued setting.
Lin and Wang \cite{LW14} generalized the unique continuation results from \cite{Dav14} to variable coefficient elliptic operators.

Within this article, we study the quantitative uniqueness of solutions to elliptic equations with singular lower order terms by generalizing the results described above in \eqref{OofV} and \eqref{UCiny} to the setting where $W \in L^s$ and $V \in L^t$ for some $s, t \le \iny$.
That is, assuming that for some $K, M \ge 1$, $\norm{W}_{L^s\pr{B_{10}}} \le K$ and $\norm{V}_{L^t\pr{B_{10}}} \le M$, let $u : B_{10} \to \C$ be a bounded, normalized solution to \eqref{ePDE} in $B_{10}$.
We show that
\begin{equation}
\norm{u}_{L^\iny\pr{B_r}} \ge c r^{C\pr{K^\kappa + M^{\mu}}} \quad \text{ as } r \to 0,
\label{OofVst}
\end{equation}
where $\kappa$ and $\mu$ depend on $s$, $t$, and $n$.
Then, using the maximal order of vanishing estimates, we employ a scaling technique to prove unique continuation at infinity theorems.
Specifically, we show that if $\norm{W}_{L^s\pr{\R^n}} \le A_1$, $\norm{V}_{L^t\pr{\R^n}} \le A_0$, and $u : \R^n \to \C$ is a bounded, normalized solution to \eqref{ePDE} in $\R^n$, then for all $R$ sufficiently large,
\begin{equation}
\mathcal{M}\pr{R}  \ge \exp\pr{-C R^\Pi \log R},
\label{UCinyst}
\end{equation}
where $\Pi$ depends on $s$, $t$, and $n$.
The precise statements of our theorems are given in the next section.

We recall some of the vast literature regarding strong unique continuation for elliptic equations with lower order terms.
Jerison and Kenig \cite{JK85} proved that the strong unique continuation property holds for operators of the form $\LP+ V$ provided that $V \in L^{n/2}_{loc}\pr{\R^n}$ for $n \ge 3$.
For operators of the form $\LP + W \cdot \gr$, Jerison in \cite{Jer86} and then Kim proved in \cite{Kim89} that strong unique continuation holds whenever $W\in L^{s}$ with $s=\frac{3n-2}{2}$, $n \ge 3$.
Further reductions of $s$ are due to Wolff in \cite{Wol90} and Regbaoui in \cite{Reg99}.
For general elliptic operators of the form $\LP + W \cdot \gr + V$, if the lower order terms satisfy $V\in L^{n/2}$ and $W\in L^{s}$ with $s>n,$ then the strong unique continuation property holds, see. e.g. \cite{KT01}.
Therefore, there is a large class of elliptic operators for which we can study quantitative uniqueness.

Vanishing order plays an important role in the study of nodal sets of eigenfunctions in geometry.
Let $\mathbb{M}$ denote a smooth, compact Riemannian manifold.
For the classical eigenfunctions of $\mathbb{M}$, those $\phi_\la$ for which
 $$-\LP_{\mathbb{M}} \phi_\lambda=\lambda \phi_\lambda \quad \quad \mbox{in} \ \mathbb{M},$$
Donnelly and Fefferman in \cite{DF88}, \cite{DF90} showed that the maximal vanishing order of $\phi_\lambda$ on $\mathbb{M}$ is everywhere less than $C\sqrt{\lambda}$, where $C$ depends only on the manifold $\mathbb{M}$.
The sharpness of this estimate is established by spherical harmonics on the sphere.
In \cite{Kuk98}, Kukavica studied the vanishing order of solutions to the Schr\"{o}dinger equation
\begin{equation}
- \LP_{\mathbb{M}} u = V u. \label{schro}
\end{equation}
Kukavica showed that if $V \in W^{1, \infty}$, then the upper bound
for the vanishing order is less than
$C\pr{1+\norm{V_-}_{L^\iny\pr{\mathbb{M}}}^{1/2} +
\osc_{\mathbb{M}}V +
\norm{\gr_{\mathbb{M}}u}_{L^\iny\pr{\mathbb{M}}}}$. Using different
methods, Bakri \cite{Bak12} and Zhu \cite{Zhu16} (when $\mathbb{M}$
is Euclidean) independently proved that the optimal vanishing order
of solutions to (\ref{schro}) is less than $\disp
C\pr{1+\norm{V}_{L^\iny\pr{\mathbb{M}}}^{1/2} +
\norm{\gr_{\mathbb{M}}V}_{L^\iny\pr{\mathbb{M}}}^{1/2}}$. The
optimality of this result can be observed if $V(x)$ is an eigenvalue
and $u(x)$ is an eigenfunction on a sphere.

In \cite{KT16}, Klein and Tsang studied quantitative unique continuation properties of (real-valued) solutions to $\LP u + V u = 0$, where $V \in L^t + L^\iny$ for some $t \ge n \ge 3$.
They used an $L^2$ Carleman estimate (similar to those that appeared in \cite{BK05}, \cite{Ken07}, and \cite{Dav14}) in combination with Sobolev embedding to derive lower bounds for solutions on small balls.
The results in \cite{KT16} imply that if $\norm{V}_{L^t} \le M$, then $$\norm{u}_{L^\iny\pr{B_r}} \ge c r^{C M^{\frac{2t}{3t-2n}}} \quad \text{ as } r \to 0.$$
It appears that the methods in \cite{KT16} do not apply when there is a singular first order term, i.e. $W \in L^s$ for some $s < \iny$.
In the present paper, through the application of more sophisticated Carleman estimates, we work with singular first and zeroth order terms (both $W$ and $V$), and we can treat $V \in L^t$ for some $t < n$.
Moreover, our bounds are smaller than those that appear in \cite{KT16}, so they may be considered stronger.

A closely related problem was studied by Kenig and Wang in \cite{KW15} where they proved vanishing order estimates for solutions to $\LP u + W \cdot \gr u = 0$ in the plane under the assumption that $W: \R^2 \to \R^2$ belongs to $L^s\pr{\R^2}$ for some $s \in \bp{2, \iny}$.
They also derived unique continuation at infinity theorems with the usual scaling technique.
The proofs in \cite{KW15} build on the complex analytic tools that were developed in \cite{KSW15}, and are therefore only suited to real-valued solutions in the plane, a setting that is very different from ours.

Finally, we point out that in \cite{MV12}, Malinnikova and Vessella studied a different quantitative uniqueness problem for elliptic operators with singular lower order terms.
They derived estimates for the norms of solutions on arbitrary compact subsets of the domain from information about the smallness of solutions on subsets of positive measure.

To prove our vanishing order estimates, Carleman estimates are used to derive three-ball inequalities.
Then the ``propagation of smallness" argument is used to obtain maximal order of vanishing
estimates.
In much of the literature discussed above, $L^2 \to L^2$ Carleman estimates were used to prove maximal order of vanishing estimates.
In this paper, we establish new $L^p \to L^q$ Carleman inequalities with $\frac{2n}{n+2}<p \leq 2 \le q \le \frac{2n}{n-2}$.
These Carleman estimates are quantitative in the sense that we show the dependence on $\tau$, the constant that may be made arbitrarily large.
The ranges of $p$ and $q$, in combination with the H\"older's inequality, allow us to consider equations of the form \eqref{ePDE}, where $W  \in L^{s}$ and $V \in L^{t}$ for a large range of $s, t < \iny$.

To verify our Carleman estimates, we decompose the Laplacian into
first order operators and prove a collection of Carleman estimates
for these operators. The $L^2 \to L^2$ Carleman estimates are proved
using the standard integration by parts approach. For the $L^p \to
L^2$ estimates, we use the eigenfunction estimates of Sogge
\cite{Sog86} along with the techniques developed in \cite{Jer86},
\cite{BKRS88} and \cite{Reg99}. By combining Carleman estimates for
the first order constituents of $\LP$, applying a Sobolev
inequality, and interpolating, we arrive at the general Carleman
estimate given in Theorem \ref{Carlpq}.

Once we have the general Carleman estimates, the order of vanishing results are proved in much the same way as in \cite{BK05} and \cite{Ken07}, for example.
The unique continuation at infinity theorems follow from the maximal order of vanishing estimates through the scaling argument presented in \cite{BK05}.

The outline of the paper is as follows.
In Section \ref{Results}, we present the precise statements of our theorems.
Section \ref{CarlEst} is devoted to obtaining Carleman estimates for the second order elliptic operators with singular lower order terms.
In Section \ref{CarlProofs}, the major $L^p\to L^q$ Carleman estimates for the Laplacian are established and we derive a quantitative Caccioppoli inequality.
In section \ref{vanOrd}, we deduce three-ball inequalities from the Carleman estimates.
Then, the vanishing order is obtained via the propagation of smallness argument.
The scaling argument is presented in Section \ref{QuantUC} where we prove the quantitative unique continuation at infinity theorems.
The letters $c$ and $C$ denote generic positive constants that do not depends on $u$, and may vary from line to line.

\section{Statements of Results}
\label{Results}

Now we present the precise statements of our theorems.
Our theorems come in pairs; the first theorem in the pair is an order of vanishing result as in \eqref{OofVst}, and the second theorem is a unique continuation at infinity estimate like \eqref{UCinyst}.
There are three pairs of theorems corresponding to the cases where $V, W \not\equiv 0$, $V \equiv 0$, and $W \equiv 0$.

Before stating the theorems, we clarify the meaning of {\em solution}.
For some $s > n$ and $t > \frac n 2$, assume that $W \in L^s\pr{B_R}$ and $V \in L^t\pr{B_R}$.
Suppose $u$ is a non-trivial solution to
\begin{equation}
\LP u+ W(x)\cdot \nabla u+V(x) u=0.
\label{goal}
\end{equation}
A priori, we assume that $u \in W^{1,2}_{loc}\pr{B_R}$ is a weak solution to \eqref{goal} in $B_R$.
However, the computations within Section \ref{CarlEst} imply that there exists a $p \in \pb{\frac{2n}{n+2}, 2}$, depending on $s$ and $t$, such that $W \cdot \nabla u+V u \in L^p_{loc}\pr{B_R}$.
By regularity theory, it follows that $u \in W_{loc}^{2,p}\pr{B_R}$ and therefore $u$ is a solution to \eqref{goal} almost everywhere in $B_R$.
Moreover, by de Giorgi-Nash-Moser theory, we have that $u \in L^\iny_{loc}\pr{B_R}$.
Therefore, when we say that $u$ is a solution to \eqref{goal} in $B_R$, it is understood that $u$ belongs to $L^\iny_{loc}\pr{B_R} \cap W^{1,2}_{loc}\pr{B_R} \cap W^{2,p}_{loc}\pr{B_R}$ and $u$ satisfies equation \eqref{goal} almost everywhere in $B_R$.

We state the first order of vanishing result.

\begin{theorem}
Let $s \in \pb{\frac{3n-2}{2}, \iny}$ and $t \in \pb{ n\pr{\frac{3n-2}{5n-2}}, \iny}$.
Assume that for some $K, M \ge 1$, $\norm{W}_{L^s\pr{B_{10}}} \le K$ and $\norm{V}_{L^t\pr{B_{10}}} \le M$.
Let $u : B_{10} \to \C$ be a solution to \eqref{goal} in $B_{10}$.
Assume that $u$ is bounded and normalized in the sense that
\begin{align}
& \|u\|_{L^\infty(B_{6})}\leq \hat{C},
\label{bound} \\
& \|u\|_{L^\infty(B_{1})}\geq 1.
\label{normal}
\end{align}
Then the maximal order of vanishing for $u$ in $B_{1}$ is less than
$ C_1 K^\kappa + C_2 M^\mu$. That is, for any $x_0\in B_1$,
\begin{align*}
\|u\|_{L^\iny(B_{r}(x_0))} &\ge c r^{\pr{ C_1 K^\kappa + C_2 M^\mu}}  \quad \text{ as } r \to 0,
\end{align*}
where
$\disp \kappa = \left\{\begin{array}{ll}
\frac{4s}{2s - \pr{3n-2}} & t > \frac{sn}{s+n} \medskip \\
\frac{4t}{\pr{5 - \frac 2 n}t - \pr{3n-2}} & n\pr{\frac{3n-2}{5n-2}}
< t \le \frac{sn}{s+n}
\end{array}\right.$,
$\disp \mu = \left\{\begin{array}{ll}
\frac{4s}{6s - \pr{3n-2}} & t \ge s \medskip \\
\frac{4 s t}{6 s t + \pr{n+2}t -4ns} & \frac{sn}{s+n} < t < s \medskip \\
\frac{4t}{\pr{5 - \frac 2 n}t - \pr{3n-2}} & n\pr{\frac{3n-2}{5n-2}}
< t \le \frac{sn}{s+n}
\end{array}\right.$,
$c = c\pr{n, s, t, \hat C}$, $C_1 = C_1\pr{n, s, t}$, and $C_2 = C_2\pr{n, s, t}$. \label{thh}
\end{theorem}

\begin{remark}
We have that if $s = \iny$, then $\disp \lim_{t \to \iny} \mu = \frac 2 3$.
And if $t > 2$, then $\disp \lim_{s \to \iny} \kappa = 2$.
Therefore, this theorem, in a sense, recovers the results described in \eqref{OofV}.
\end{remark}
\begin{remark}
Because of the dependence on both $W$ and $V$ in the vanishing order of solutions, the powers $\kappa$ and $\mu$ in Theorem \ref{thh} depend on both $t$ and $s$, and therefore the ranges of $t$ and $s$ are correlated.
\end{remark}
\begin{remark}
If we compare Theorem \ref{thh} in the case $s=\infty$ to Theorem \ref{thhh} below, we see that the power $\mu$ is smaller in Theorem \ref{thhh}.
\end{remark}

As in \cite{BK05}, a scaling argument shows that the following unique continuation at infinity estimate follows from Theorem \ref{thh}.
Each unique continuation at infinity theorem is presented in terms of a lower bound for $\mathcal{M}\pr{R}$, where
\begin{equation}
\mathcal{M}\pr{R} := \inf\set{\norm{u}_{L^\iny\pr{B_1\pr{x_0}}} : \abs{x_0} = R}.
\label{MRDef}
\end{equation}
(Compare with the estimate given in \eqref{UCiny}.)

\begin{theorem}
Assume that for some $s \in \pb{\frac{3n-2}{2}, \iny}$ and $t \in \pb{ n\pr{\frac{3n-2}{5n-2}}, \iny}$, $\norm{W}_{L^s\pr{\R^n}} \le A_1$ and $\norm{V}_{L^t\pr{\R^n}} \le A_0$.
Let $u : \R^n \to \C$ be a solution to \eqref{goal} in $\R^n$.
Assume that $\norm{u}_{L^\iny\pr{\R^n}} \le C_0$ and $\abs{u\pr{0}} \ge 1$.
Then for $R >> 1$,
\begin{equation*}
\mathcal{M}\pr{R} \ge \exp\pr{-C R^\Pi \log R},
\end{equation*}
where
$\disp
\Pi = \left\{\begin{array}{ll}
\frac{4\pr{s-n}}{2s - \pr{3n-2}} & t > \frac{sn}{s+n} \medskip \\
\frac{4 \pr{t-n\frac t s}}{\pr{5 - \frac 2 n}s - \pr{3n-2} \frac s t} & n\pr{\frac{3n-2}{5n-2}} < t \le \frac{sn}{s+n}
\end{array}\right.
$, and $C = C\pr{n, s, t, A_1, A_0, C_0}$.
\label{UCVW}
\end{theorem}

Now we consider solutions to equation \eqref{goal} with $V(x)\equiv0$, i.e. solutions to
\begin{equation}
\LP u+W(x) \cdot \nabla u=0.
\label{drift}
\end{equation}
An immediate consequence of Theorem \ref{thh} is the vanishing order of solutions for second order elliptic equations with drift.

\begin{corollary}
Let $s \in \pb{\frac{3n-2}{2}, \iny}$. Assume that for some $K \ge
1$, $\norm{W}_{L^s\pr{B_{10}}} \le K$. Let $u : B_{10} \to \C$ be a
solution to \eqref{drift} in $B_{10}$ that is bounded and normalized
in the sense of \eqref{bound} and \eqref{normal}. Then the maximal
order of vanishing for $u$ in  $B_{1}$ is less than $C_1
K^\kappa$. That is,  for any $x_0\in B_1$, 
\begin{align*}
 \|u\|_{L^\iny(B_{r}(x_0))} &\ge c r^{C_1 K^\kappa} \quad \text{ as } r \to 0,
\end{align*}
where $\disp \kappa = \frac{4s}{2s - \pr{3n-2}}$, $c = c\pr{n, s, \hat C}$, and $C_1 = C_1\pr{n, s}$.
\label{thhCor}
\end{corollary}

The following unique continuation estimate follows from Corollary \ref{thhCor} in the same way that Theorem \ref{UCVW} follows from Theorem \ref{thh}.

\begin{corollary}
Assume that $\norm{W}_{L^s\pr{\R^n}} \le A_1$ for some $s \in \pb{\frac{3n-2}{2}, \iny}$.
Let $u : \R^n \to \C$ be a solution to \eqref{drift} in $\R^n$ for which $\norm{u}_{L^\iny\pr{\R^n}} \le C_0$ and $\abs{u\pr{0}} \ge 1$.
Then for $R >> 1$,
\begin{equation*}
\mathcal{M}\pr{R} \ge \exp\pr{-C R^\Pi \log R}
\end{equation*}
where $\disp \Pi = \frac{4\pr{s-n}}{2s - \pr{3n-2}}$, and $C = C\pr{n, s, A_1, C_0}$. \label{UCW}
\end{corollary}

Finally, we consider solutions to an elliptic equation without a gradient potential,
\begin{equation}
\LP u+V(x) u=0 .
\label{goal1}
\end{equation}

\begin{theorem}
Let $t \in \pb{\frac{4 n^2}{7n+2}, \iny}$. Assume that for some $M
\ge 1$, $\norm{V}_{L^t\pr{B_{R_0}}} \le M$. Let $u : B_{10} \to \C$
be a solution to \eqref{goal1} in $B_{10}$ that is bounded and
normalized in the sense of \eqref{bound} and \eqref{normal}. Then
the maximal order of vanishing for $u$ in  $B_{1}$ is less than $C_2M^\mu$. That is, for any $x_0\in B_1$, 
\begin{align*}
 \|u\|_{L^\iny(B_{r}(x_0))} &\ge c r^{C_2M^\mu} \quad \text{ as } r \to 0,
\end{align*}
where for any positive $\eps < \min\set{\frac{7t+\frac{2t}n-4n}{2}, \frac{(2t-n)(n+2)}{2n} }$,
$\disp \mu = \left\{\begin{array}{ll}
\frac{4t}{6t - \pr{3n-2}} & t > n \medskip \\
\frac{4t} {7t+\frac{2t}n-4n-\eps} & \frac{4n^2}{7n+2} < t \le n
\end{array}\right.$,
$c=c\pr{n, t, \hat C}$, and $C_2 = C_2\pr{n, t, \eps}$.
\label{thhh}
\end{theorem}

Using the maximal order of vanishing estimate from the previous theorem, we may prove quantitative unique continuation at infinity estimates.
Notice that the value of $\Pi$ here is much smaller than the one in Theorem \ref{UCVW}.

\begin{theorem}
Assume that $\norm{V}_{L^t\pr{\R^n}} \le A_0$ for some $t \in \pb{ \frac{4 n^2}{7n+2}, \iny}$.
Let $u : \R^n \to \C$ be a solution to \eqref{goal1} in $\R^n$ for which $\norm{u}_{L^\iny\pr{\R^n}} \le C_0$ and $\abs{u\pr{0}} \ge 1$.
Then for $R >> 1$,
\begin{equation*}
\mathcal{M}\pr{R} \ge \exp\pr{-C R^\Pi \log R},
\end{equation*}
where for any positive $\eps < \min\set{\frac{7t+\frac{2t}n-4n}{2}, \frac{(2t-n)(n+2)}{2n} }$,
$\disp \Pi =  \left\{\begin{array}{ll}
\frac{4\pr{2t-n}}{6t - \pr{3n-2}} & t > n \medskip \\
\frac{4\pr{2t-n}} {7t+\frac{2t}n-4n-\eps} & \frac{4n^2}{7n+2} < t \le n
\end{array}\right.$,
 and $C = C\pr{n, t, A_0, C_0, \eps}$.
\label{UCV}
\end{theorem}

\section{Carleman estimates }
\label{CarlEst}

In this section, we state the crucial tools, the quantitative $L^p-L^{ q}$ type Carleman estimates.
Let $r=|x-x_0|$ and set
$$\phi(r)=\log r+\log(\log r)^2.$$
We use the notation $\|u\|_{L^p(r^{-n} dx)}$ to denote the $L^p$ norm with weight $r^{-n}$, i.e. $\|u\|_{L^p(r^{-n} dx)} = \disp \pr{\int |u\pr{x}|^p r^{-n}\, dx}^\frac{1}{p}$.
Our quantitative $L^p - L^q$ Carleman estimate for the Laplacian is as follows.

\begin{theorem}
Let $\frac{2n}{n+2} < p \le 2 \leq q \leq \frac{2n}{n-2}$. There
exists a constant $C$, depending on $n$, $p$, and $q$, and a sufficiently small $R_0$ such that for any
$u\in C^{\infty}_{0}\pr{B_{R_0}(x_0)\backslash\set{x_0} }$ and
$\tau>1$, one has
\begin{align}
&\tau^{\be_0} \|(\log r)^{-1} e^{-\tau \phi(r)}u\|_{L^q(r^{-n}dx)} +
\tau^{\be_1} \|(\log r )^{-1} e^{-\tau \phi(r)}r \nabla
u\|_{L^2(r^{-n}dx)}
\nonumber \medskip\\
&\leq  C \|(\log r ) e^{-\tau \phi(r)} r^2 \LP u\|_{L^p(r^{-n} dx)} ,
\label{mainCar}
\end{align}
where $\be_0 = \frac 3 2 -
\frac{\pr{3n-2}\pr{2-p}}{8p}-\frac{n(q-2)}{2q}$ and $\be_1 =
\frac{1}{2}-\frac{\pr{3n-2}\pr{2-p}}{8p}$.
 \label{Carlpq}
\end{theorem}

\begin{remark}
The Carleman estimates given by \eqref{mainCar} differ from those in \cite{JK85} and \cite{KT01} where $p=\frac{2n}{n+2}$ and $q=\frac{2n}{n-2}$ since our estimates in \eqref{mainCar} hold for a range of $p$ and $q$ values.
Furthermore, the power of the parameter $\tau$ is shown explicitly, which is crucial to calculating the vanishing order of solutions.
\end{remark}
\begin{remark}
We may at times use the notation $\be_0\pr{q}$ to remind the reader that $\be_0$ depends on $q$.
This notation will be useful when we work with multiple $q$-values.
\end{remark}

The proof of Theorem \ref{Carlpq} is given in the next section.

Now we use Theorem \ref{Carlpq} to establish the following $L^p\to L^2$ Carleman estimates for second order elliptic equations of the form \eqref{goal}.
We show that for an appropriate choice of $p$, and for $\tau$ sufficiently large, we may replace the Laplacian with a more general elliptic operator.
In each of the following three theorems, we use H\"older's inequality and the triangle inequality to go from Theorem \ref{Carlpq} to a Carleman estimate for an elliptic operator with lower order terms.
Since the argument is the simplest, we start with drift operators ($V \equiv 0$) corresponding to equations of the form \eqref{drift}.

\begin{theorem}
Let $s \in \pb{\frac{3n-2}{2}, \iny}$.
Assume that for some $K \ge 1$, $\norm{W}_{L^s\pr{B_{R_0}}} \le K$.
Then there exist constants $C_0$, $C_1$, and sufficiently small $R_0 < 1$  such that for any $u\in C^{\infty}_{0}(B_{R_0}(x_0)\setminus \set{x_0})$ and large positive constant
$$\tau \ge C_1 K^{\kappa},$$
one has
\begin{align}
\tau^{\be_0} \|(\log r)^{-1} e^{-\tau \phi(r)}u\|_{L^2(r^{-n}dx)}
&\leq  C_0 \|(\log r ) e^{-\tau \phi(r)} r^2\pr{ \LP u + W \cdot \gr u}\|_{L^p(r^{-n} dx)} ,
\label{main1}
\end{align}
where $\kappa = \frac{4s}{2s - \pr{3n-2}}$, $p = \frac{2s}{s+2}$, and $\be_0 = \be_0\pr{2}$ as defined in Theorem \ref{Carlpq}.
Moreover, $C_0 = C$ from Theorem \ref{Carlpq}, and $C_1 = C_1\pr{n, s}$.
\label{CarlpqW}
\end{theorem}

\begin{proof}
By \eqref{mainCar} in Theorem \ref{Carlpq} with $q = 2$ and the triangle inequality,
\begin{align}
&\tau^{\be_0} \|(\log r)^{-1} e^{-\tau \phi(r)}u\|_{L^2(r^{-n}dx)}
+\tau^{\be_1} \|(\log r )^{-1} e^{-\tau \phi(r)}r \nabla u\|_{L^2(r^{-n}dx)} \nonumber \\
&\leq  C \|(\log r ) e^{-\tau \phi(r)} r^2 \LP u\|_{L^p(r^{-n} dx)}  \nonumber \\
&\le C\|(\log r) e^{-\tau \phi(r)} r^2 (\LP u+ W\cdot \nabla u)\|_{L^p(r^{-n} dx)}
+ C\|(\log r) e^{-\tau \phi(r)} r^2 W\cdot \nabla u \|_{L^p(r^{-n} dx)}.
\label{triIneq0}
\end{align}
Therefore, to reach the conclusion of the theorem, we need to absorb the second term on the right of \eqref{triIneq0} into the lefthand side.

Set $p = \frac{2s}{s+2}$.
By assumption, $s > \frac{3n-2}{2} > n$, so it follows that $p \in \pb{\frac{2n}{n+2}, 2}$.
Since $\frac{1}{p}=\frac{1}{s}+\frac{1}{2}$, then by an application of H\"older's inequality we have
\begin{align}
&\|(\log r) e^{-\tau\phi(r)} r^2 W\cdot\nabla u\|_{L^p(r^{-n} dx)} \nonumber \\
&\le  \|W\|_{L^{s}\pr{B_{R_0}}} \|(\log r) e^{-\tau \phi(r)} r^{-\frac{n}{p}+\frac{n}{2}+2} \nabla u \|_{L^2(r^{-n} dx)} \nonumber \\
&\le  \|W\|_{L^{s}\pr{B_{R_0}}} \|\pr{\log r}^2 r^{1 + \frac n 2 - \frac n p}\|_{L^{\iny}\pr{B_{R_0}}} \|(\log r)^{-1} e^{-\tau \phi(r)} r \nabla u \|_{L^2(r^{-n} dx)} \nonumber \\
&\le c K \|(\log r)^{-1} e^{-\tau \phi(r)} r \nabla u \|_{L^2(r^{-n} dx)},
\label{hod1}
\end{align}
where $p \in \pb{\frac{2n}{n+2}, 2}$ implies that $1 + \frac n 2 - \frac n p > 0$, so that $\pr{\log r}^2 r^{1 + \frac n 2 - \frac n p}$ is bounded on $B_{R_0}$.

By combining \eqref{triIneq0} and \eqref{hod1}, we see that
\begin{align*}
&\tau^{\be_0} \|(\log r)^{-1} e^{-\tau \phi(r)}u\|_{L^2(r^{-n}dx)}
+\tau^{\be_1} \|(\log r )^{-1} e^{-\tau \phi(r)}r \nabla u\|_{L^2(r^{-n}dx)} \nonumber \\
&\le C\|(\log r) e^{-\tau \phi(r)} r^2 (\LP u+ W\cdot \nabla u)\|_{L^p(r^{-n} dx)}
+ c C K \|(\log r)^{-1} e^{-\tau \phi(r)} r \nabla u \|_{L^2(r^{-n} dx)}.
\end{align*}
Since $\be_1 = \frac 1 2 - \frac{3n-2}{4s} = \frac{2s - \pr{3n-2}}{4s} > 0$, if we choose $\tau \ge  \pr{c C K}^{\frac 1 {\be_1}}$, then we may absorb the second term on the right into the lefthandside.
That is,
\begin{align*}
\tau^{\be_0} \|(\log r)^{-1} e^{-\tau \phi(r)}u\|_{L^2(r^{-n}dx)} 
&\le C\|(\log r) e^{-\tau \phi(r)} r^2 (\LP u+ W\cdot \nabla u)\|_{L^p(r^{-n} dx)},
\end{align*}
giving the conclusion of the theorem.
\end{proof}

Now we consider more general elliptic operators.

\begin{theorem}
Let $s \in \pb{\frac{3n-2}{2}, \iny}$ and $t \in \pb{ n\pr{\frac{3n-2}{5n-2}}, \iny}$.
Assume that for some $K, M \ge 1$, $\norm{W}_{L^s\pr{B_{R_0}}} \le K$ and $\norm{V}_{L^t\pr{B_{R_0}}} \le M$.
Then there exist constants $C_0$, $C_1$, $C_2$, and sufficiently small $R_0 < 1$  such that for any $u\in C^{\infty}_{0}(B_{R_0}(x_0)\setminus \set{x_0})$ and large positive constant
$$\tau \ge C_1 K^{\kappa} + C_2 M^{\mu},$$
one has
\begin{align}
\tau^{\be_0} \|(\log r)^{-1} e^{-\tau \phi(r)}u\|_{L^2(r^{-n}dx)}
&\leq  C_0 \|(\log r ) e^{-\tau \phi(r)} r^2\pr{ \LP u + W \cdot \gr u + V u}\|_{L^p(r^{-n} dx)} ,
\label{main1}
\end{align}
where
$\disp \kappa = \left\{\begin{array}{ll}
\frac{4s}{2s - \pr{3n-2}} & t > \frac{sn}{s+n} \medskip \\
\frac{4t}{\pr{5 - \frac 2 n}t - \pr{3n-2}} & n\pr{\frac{3n-2}{5n-2}}
< t \le \frac{sn}{s+n}
\end{array}\right.$,
$\disp \mu = \left\{\begin{array}{ll}
\frac{4s}{6s - \pr{3n-2}} & t \ge s \medskip \\
\frac{4 s t}{6 s t + \pr{n+2}t -4ns} & \frac{sn}{s+n} < t < s \medskip \\
\frac{4t}{\pr{5 - \frac 2 n}t - \pr{3n-2}} & n\pr{\frac{3n-2}{5n-2}}
< t \le \frac{sn}{s+n}
\end{array}\right.$,
$\disp p = \left\{\begin{array}{ll}
\frac{2s}{s+2} & t > \frac{sn}{s+n} \medskip \\
\frac{2 n t }{2n - 2t + nt } & n\pr{\frac{3n-2}{5n-2}} < t \le
\frac{sn}{s+n}
\end{array}\right.$,
and $\be_0 = \be_0\pr{2}$ as defined in Theorem \ref{Carlpq}.
Moreover, $C_0 = 2C$, where $C$ is from Theorem \ref{Carlpq}, $C_1 = C_1\pr{n, s, t}$, and $C_2 = C_2\pr{n,s,t}$.
\label{CarlpqVW}
\end{theorem}

\begin{proof}
If we add inequality \eqref{mainCar} from Theorem \ref{Carlpq} with $q = 2$ to the same inequality with $q$ arbitrary, we see that
\begin{align}
&\tau^{\be_0\pr{2}} \|(\log r)^{-1} e^{-\tau \phi(r)}u\|_{L^2(r^{-n}dx)}
+\tau^{\be_0\pr{q}} \|(\log r)^{-1} e^{-\tau \phi(r)}u\|_{L^q(r^{-n}dx)} \nonumber \\
&+\tau^{\be_1} \|(\log r )^{-1} e^{-\tau \phi(r)}r \nabla u\|_{L^2(r^{-n}dx)} \nonumber \\
&\le 2 C\|(\log r) e^{-\tau \phi(r)} r^2 (\LP u)\|_{L^p(r^{-n} dx)} \nonumber \\
&\le 2 C\|(\log r) e^{-\tau \phi(r)} r^2 (\LP u+ W\cdot \nabla u + V u)\|_{L^p(r^{-n} dx)} \nonumber \\
&+ 2 C\|(\log r) e^{-\tau \phi(r)} r^2 W\cdot \nabla u \|_{L^p(r^{-n} dx)}
+ 2 C\|(\log r) e^{-\tau \phi(r)} r^2 V u \|_{L^p(r^{-n} dx)},
\label{triIneq}
\end{align}
where the last line follows from the triangle inequality.
Therefore, to reach the conclusion of the lemma, we need to choose $p$ and $q$, and make $\tau$ large enough so that we may absorb the last two terms into the lefthand side.

An application of H\"older' inequality shows that if $p \in \pb{\frac{2n}{n+2}, 2}$, then
\begin{align}
& \|(\log r) e^{-\tau\phi(r)} r^2 W\cdot\nabla u\|_{L^p(r^{-n} dx)} \nonumber \medskip \\
&\le  \|W\|_{L^{\frac{2p}{2-p}}\pr{B_{R_0}}} \|\pr{\log r}^2 r^{1 + \frac n 2 - \frac n p}\|_{L^{\iny}\pr{B_{R_0}}} \|(\log r)^{-1} e^{-\tau \phi(r)} r \nabla u \|_{L^2(r^{-n} dx)} \nonumber \\
&\le c \|W\|_{L^{\frac{2p}{2-p}}\pr{B_{R_0}}} \|(\log r)^{-1} e^{-\tau \phi(r)} r \nabla u \|_{L^2(r^{-n} dx)},
\label{hod2}
\end{align}
since $1 + \frac n 2 - \frac n p > 0$.
Similarly, if we further assume that $q \ge p$, then
\begin{align}
\|(\log r) e^{-\tau \phi(r)} r^2 V u \|_{L^p(r^{-n} dx)}
&\le  \|V\|_{L^{\frac{pq}{q-p}}\pr{B_{R_0}}} \|\pr{\log r}^2 r^{2 + \frac n q - \frac n p}\|_{L^{\iny}\pr{B_{R_0}}}
\nonumber \\ &\times\|(\log r)^{-1} e^{-\tau \phi(r)} u\|_{L^q(r^{-n} dx)} \nonumber \\
&\le c \|V\|_{L^{\frac{pq}{q-p}}\pr{B_{R_0}}} \|(\log r)^{-1} e^{-\tau \phi(r)} u\|_{L^q(r^{-n} dx)}.
\label{hod3}
\end{align}

{\bf Case 1: $t \in \brac{s, \iny}$} \\
If $t \ge s$, then we choose $p = \frac{2s}{s+2}$ and $q = 2$.
Since $s > \frac{3n-2}{2}$, then $p$ is in the appropriate range.
As $\frac{2p}{2-p} = \frac{pq}{q-p} = s \le t$, then substituting \eqref{hod2} and \eqref{hod3} into \eqref{triIneq}, and using that $\|V\|_{L^{s}} \le c \|V\|_{L^{t}}$ by H\"older's inequality, we have that
\begin{align*}
&2\tau^{\be_0\pr{2}} \|(\log r)^{-1} e^{-\tau \phi(r)}u\|_{L^2(r^{-n}dx)}
+\tau^{\be_1} \|(\log r )^{-1} e^{-\tau \phi(r)}r \nabla u\|_{L^2(r^{-n}dx)} \nonumber \\
&\le C\|(\log r) e^{-\tau \phi(r)} r^2 (\LP u+ W\cdot \nabla u + V u)\|_{L^p(r^{-n} dx)} \nonumber \\
&+ 2 c C K \|(\log r)^{-1} e^{-\tau \phi(r)} r \nabla u \|_{L^2(r^{-n} dx)}
+ 2 c C M \|(\log r)^{-1} e^{-\tau \phi(r)} u\|_{L^q(r^{-n} dx)}.
\end{align*}
In this case, $\be_0\pr{2} = \frac{3}{2} - \frac{3n-2}{4s}$, $\be_1 = \frac{1}{2} - \frac{3n-2}{4s}$, and the lower bound on $s$ ensures that $\be_1 > 0$.
If we choose $\tau \ge \pr{2 c C K}^{\frac 1 {\be_1}}+ \pr{2 c C M}^{\frac 1 {\be_0\pr{2}}} $, then we may absorb the last two terms on the right into the lefthandside to reach the conclusion of the theorem.

{\bf Case 2: $t \in \pr{\frac{sn}{s+n}, s}$} \\
In this case, we choose $p = \frac{2s}{s+2}$ and $q = \frac{2st}{st+2t-2s}$.
As before, $p$ falls in the appropriate range and the bounds on $t$ ensure that $q \in \pr{2, \frac{2n}{n-2}}$.
Since $\frac{2p}{2-p} = s$ and $\frac{pq}{q-p} = t$, then upon substituting \eqref{hod2} and \eqref{hod3} into \eqref{triIneq}, we see that
\begin{align}
&\tau^{\be_0\pr{2}} \|(\log r)^{-1} e^{-\tau \phi(r)}u\|_{L^2(r^{-n}dx)}
+\tau^{\be_0\pr{q}} \|(\log r)^{-1} e^{-\tau \phi(r)}u\|_{L^q(r^{-n}dx)} \nonumber \\
&+\tau^{\be_1} \|(\log r )^{-1} e^{-\tau \phi(r)}r \nabla u\|_{L^2(r^{-n}dx)} \nonumber \\
&\le 2 C\|(\log r) e^{-\tau \phi(r)} r^2 (\LP u+ W\cdot \nabla u + V u)\|_{L^p(r^{-n} dx)} \nonumber \\
&+ 2 c C K \|(\log r)^{-1} e^{-\tau \phi(r)} r \nabla u \|_{L^2(r^{-n} dx)}
+ 2 c C M \|(\log r)^{-1} e^{-\tau \phi(r)} u\|_{L^q(r^{-n} dx)}.
\label{boundSub}
\end{align}
Now $\be_0\pr{q} = \frac 3 2 + \frac{n+2}{4s} - \frac n t$ and $\be_1 =\frac{1}{2}-\frac{3n-2}{4s}$ are both positive.
We again take $\tau \ge \pr{2 c C K}^{\frac 1 {\be_1}}+ \pr{2 c C M}^{\frac 1 {\be_0\pr{q}}} $ in order to absorb the last two terms into the lefthand side.

{\bf Case 3: $t \in \pb{n\pr{\frac{3n-2}{5n-2}}, \frac{sn}{s+n}}$} \\
This time we choose $p = \frac{2nt}{2n-2t+nt}$ and $q = \frac{2n}{n-2}$.
Since $\frac n 2 < n\pr{\frac{3n-2}{5n-2}}$ and $\frac{sn}{s+n} < n$, then $p \in \pr{\frac{2n}{n+2}, 2}$.
Since $\frac{2p}{2-p} = \frac{nt}{n-t} \le s$, then an application of H\"older's inequality shows that $\|W\|_{L^{\frac{2p}{2-p}}} \le c \|W\|_{L^{s}}$.
Noting that $\frac{pq}{q-p} = t$, we again show that \eqref{boundSub} holds.
Now $\be_1 = \frac 1 2 - \frac{3n-2}{4}\pr{\frac 1 t - \frac 1 n}  = \frac {5n-2}{4n} - \frac{3n-2}{4t}$ and the lower bound on $t$ implies that $\be_1 > 0$.
Since $\be_0\pr{q} = \be_1$, then $\be_0 > 0$ as well.
If we choose $\tau \ge \pr{4 c C K}^{\frac 1 {\be_1}} + \pr{4 c C M}^{\frac 1 {\be_0\pr{q}}}$, we can absorb the last two terms into the lefthand side to reach the conclusion.
\end{proof}

Now we consider the second order elliptic equation $\LP u+V(x)u=0$.
The proof is similar to the previous one.

\begin{theorem}
Let $t > \frac{4 n^2}{7n+2}$.
Assume that for some $M \ge 1$, $\norm{V}_{L^t\pr{B_{R_0}}} \le M$.
Then there exist constants $C_0$, $C_2$, and sufficiently small $R_0 < 1$  such that for any $u\in C^{\infty}_{0}(B_{R_0}(x_0)\setminus \set{x_0})$ and large positive constant
$$\tau \ge C_2 M^{\mu},$$
one has
\begin{align}
\tau^{\be_0} \|(\log r)^{-1} e^{-\tau \phi(r)}u\|_{L^2(r^{-n}dx)}
&\leq  C_0 \|(\log r ) e^{-\tau \phi(r)} r^2\pr{ \LP u + V u}\|_{L^p(r^{-n} dx)} ,
\label{main3}
\end{align}
where for any positive $\eps < \min\set{\frac{7t+\frac{2t}n-4n}{2}, \frac{(2t-n)(n+2)}{2n} }$, we
have $\disp \mu = \left\{\begin{array}{ll}
\frac{4t}{6t - \pr{3n-2}} & t > n \medskip \\
\frac{4t} {7t+\frac{2t}n-4n-\eps} & \frac{4n^2}{7n+2} < t \le n
\end{array}\right.$,
$\disp p = \left\{\begin{array}{ll}
\frac{2t}{t+2} & t > n \medskip \\
\frac{2n}{n + 2 -  \frac{2n\eps}{\pr{n+2}t} }& \frac{4n^2}{7n+2} < t
\le n
\end{array}\right.$,
and $\be_0 = \be_0\pr{2}$ as defined in Theorem \ref{Carlpq}.
Moreover, $C_0 = 2 C$, where $C$ is from Theorem \ref{Carlpq}, and $C_2 = C_2\pr{n,t, \eps}$.
\label{CarlpqV}
\end{theorem}

\begin{proof}
As in the previous proof, if we add inequality \eqref{mainCar} from Theorem \ref{Carlpq} with $q = 2$ to the same inequality with $q$ arbitrary, we see that
\begin{align}
&\tau^{\be_0\pr{2}} \|(\log r)^{-1} e^{-\tau \phi(r)}u\|_{L^2(r^{-n}dx)}
+\tau^{\be_0\pr{q}} \|(\log r)^{-1} e^{-\tau \phi(r)}u\|_{L^q(r^{-n}dx)} \nonumber \\
&\leq 2 C \|(\log r ) e^{-\tau \phi(r)} r^2 \LP u\|_{L^p(r^{-n} dx)}  \nonumber \\
&\le 2 C\|(\log r) e^{-\tau \phi(r)} r^2 (\LP u + V u)\|_{L^p(r^{-n} dx)}
+ 2 C\|(\log r) e^{-\tau \phi(r)} r^2 V u \|_{L^p(r^{-n} dx)},
\label{triIneqV}
\end{align}
where we have used the triangle inequality to reach the last line.
We need to choose $p, q$ so that the last term on the right can be absorbed into the second term on the left while making $\tau$ minimally large.
As shown in Theorem \ref{CarlpqVW}, if $p \in \pb{\frac{2n}{n+2}, 2}$ and $q \ge p$, then
\begin{align}
\|(\log r) e^{-\tau \phi(r)} r^2 V u \|_{L^p(r^{-n} dx)}
&\le c \|V\|_{L^{\frac{pq}{q-p}}\pr{B_{R_0}}} \|(\log r)^{-1} e^{-\tau \phi(r)} u\|_{L^q(r^{-n} dx)}.
\label{hod4}
\end{align}
Again, we will work in different cases corresponding to different ranges of $t$.

{\bf Case 1: $t > n$} \\
Set $p = \frac{2t}{t+2}$ and $q = 2$.
Since $t > n$, then $p \in \left(\frac{2n}{n+2}, 2\right]$, as required.
As $\frac{pq}{q-p} = t$, then by combining \eqref{triIneqV} and \eqref{hod4}, we see that
\begin{align*}
&2\tau^{\be_0\pr{2}} \|(\log r)^{-1} e^{-\tau \phi(r)}u\|_{L^2(r^{-n}dx)}
 \nonumber \\
&\le 2 C\|(\log r) e^{-\tau \phi(r)} r^2 (\LP u + V u)\|_{L^p(r^{-n} dx)}
+ 2 c C M \|(\log r)^{-1}e^{-\tau \phi(r)} u\|_{L^q(r^{-n}dx)}.
\end{align*}
Since $\be_0\pr{2} = \frac 3 2 - \frac{3n-2}{4t} > 0$, if we ensure that $\tau \ge \pr{2 c C M}^{\frac 1 {\be_0}}$, then the second term on the right may be absorbed into the left and the conclusion of the theorem follows.

{\bf Case 2:} $t \in\pb{\frac{4n^2}{7n+2}, n}$. \\
Choose $\eps \in \pr{0, \min\set{\frac{7t+\frac{2t}n-4n}{2},
\frac{(2t-n)(n+2)}{2n}}}$ to be arbitrarily small. Set $p =
\frac{2n}{n + 2 - \frac{2n\eps}{\pr{n+2}t} }$ and take $q =
\frac{pt}{t-p}$. Since $\eps
<\frac{(2t-n)(n+2)}{n}<\frac{(n+2)t}{n}$, then $p \in
\pr{\frac{2n}{n+2}, 2}$ and $q \in \pr{2, \frac{2n}{n-2}}$. As
$\frac{pq}{q-p} = t$, then upon substituting \eqref{hod4} into
\eqref{triIneqV}, we see that
\begin{align*}
&\tau^{\be_0\pr{2}} \|(\log r)^{-1} e^{-\tau \phi(r)}u\|_{L^2(r^{-n}dx)}
+\tau^{\be_0\pr{q}} \|(\log r)^{-1} e^{-\tau \phi(r)}u\|_{L^q(r^{-n}dx)} \nonumber \\
&\le 2 C\|(\log r) e^{-\tau \phi(r)} r^2 (\LP u + V u)\|_{L^p(r^{-n} dx)}
+ 2 c C M \|(\log r)^{-1} e^{-\tau \phi(r)} u\|_{L^q(r^{-n} dx)}
\label{triIneqV}
\end{align*}
Now $\be_0\pr{q} = \frac {7t+\frac{2t}n-4n-\eps}{4t} >0$ (since $2\eps <7t+\frac{2t}n-4n$) and if we choose $\tau \ge \pr{2 c C M}^{\frac 1 {\be_0\pr{q}}}$, the conclusion of the theorem follows.
\end{proof}

\section{Proof of $L^p\to L^{q}$ Carleman estimates }
\label{CarlProofs}

In this section, we prove the crucial tool in the whole paper, i.e. the $L^p - L^q$ Carleman estimate stated in Theorem \ref{Carlpq}.
To prove our Carleman estimate, we first establish some intermediate Carleman estimates for first-order operators.

We introduce polar coordinates in $\mathbb R^n\backslash \{0\}$ by setting $x=r\omega$, with $r=|x|$ and $\omega=(\omega_1,\cdots,\omega_n)\in S^{n-1}$.
Further, we use a new coordinate $t=\log r$.
Then
$$ \frac{\partial }{\partial x_j}=e^{-t}(\omega_j\partial_t+  \Omega_j), \quad 1\leq j\leq n, $$
where $\Omega_j$ is a vector field in $S^{n-1}$.
It is well known that vector fields $\Omega_j$ satisfy
$$ \sum^{n}_{j=1}\omega_j\Omega_j=0 \quad \mbox{and} \quad
\sum^{n}_{j=1}\Omega_j\omega_j=n-1.$$
In the new coordinate system, the Laplace operator takes the form
\begin{equation}
e^{2t} \LP=\partial^2_t u+(n-2)\partial_t+\LP_{\omega},
\label{laplace}
\end{equation}
where $\disp \LP_\omega=\sum_{j=1}^n \Omega^2_j$ is the Laplace-Beltrami operator on $S^{n-1}$.
The eigenvalues for $-\LP_\omega$ are $k(k+n-2)$, $k\in \mathbb{N}$, where $\mathbb{N}$ denotes the set of nonnegative integers.
The corresponding eigenspace is $E_k$, the space of spherical harmonics of degree $k$.
It follows that
$$\| \LP_\omega v\|^2_{L^2(dtd\omega)}=\sum_{k\geq 0} k^2(k+n-2)^2\| v_k\|^2_{L^2(dtd\omega)}$$
and
\begin{equation}
\sum_{j=1}^n\| \Omega_j v\|^2_{L^2(dtd\omega)}
=\sum_{k\geq 0} k(k+n-2)\|v_k\|^2_{L^2(dtd\omega)},
\label{lll}
\end{equation}
where $v_k$ denotes the projection of $v$ onto $E_k$.
Here $\|\cdot\|_{L^2(dtd\omega)}$ denotes the $L^2$ norm on $(-\infty, \infty)\times S^{n-1}$.

 Let
$$\Lambda=\sqrt{\frac{(n-2)^2}{4}-\LP_\omega}.$$ The operator $\Lambda$ is a
first-order elliptic pseudodifferential operator on $L^2(S^{n-1})$.
The eigenvalues for the operator $\Lambda$ are $k+\frac{n-2}{2}$, with corresponding eigenspace $E_k$.
That is, for any $v\in C^\infty_0(S^{n-1})$,
\begin{equation}
\pr{\Lambda - \frac{n-2}{2}} v= \sum_{k\geq 0}k P_k v,
\label{ord}
\end{equation}
where $P_k$ is the projection operator from $L^2(S^{n-1})$ onto $E_k$.
We remark that the projection operator, $P_k$, acts only on the angular variables.
In particular, $P_k v\pr{t, \om} = P_k v\pr{t, \cdot} \pr{\om}$.

Now define
\begin{equation} L^\pm=\partial_t+\frac{n-2}{2}\pm \Lambda.
\label{use}
\end{equation} From the equation (\ref{laplace}), it follows that
\begin{equation*}
e^{2t} \LP=L^+L^-=L^-L^+.
\end{equation*}

Recall that we introduced the weight function
$$\phi(r)=\log r+\log(\log r)^2.$$
With $r=e^t$, define the weight function in terms of $t$ to be
$$\varphi(t)=\phi(e^t)=t+\log t^2.$$
We are only interested in those values of $r$ that are sufficiently small.
Since $r\to 0$ if and only if $t\to-\infty$ then, in terms of the new coordinate $t$, we study the case when $t$ is sufficiently close to $-\infty$.

We first establish an $L^2- L^2$ Carleman inequality for the operator $L^+$.

\begin{lemma}
If $\abs{t_0}$ is sufficiently large, then for any $v \in C^\iny_0\pr{\pr{-\iny, -\abs{t_0}} \times S^{n-1}}$, we have that
\begin{eqnarray}
\tau \norm{t^{-1} e^{-\tau \varphi(t)}v}_{L^2(dtd\omega )}
&+&\norm{t^{-1}e^{-\tau \varphi(t)} \partial_t v}_{L^2(dtd\omega )}
+\sum_{j=1}^n \norm{t^{-1}e^{-\tau \varphi(t)} \Omega_j v }_{L^2(dtd\omega )}   \nonumber \medskip\\
&\leq& C\norm{t^{-1} e^{-\tau \varphi(t)} L^+v}_{L^2(dtd\omega )},
\label{cond}
 \end{eqnarray}
\label{Car22}
where $C$ is a universal constant.
\end{lemma}

\begin{proof}
Recall that $L^+=\partial_t+\Lambda+\frac{n-2}{2}$.
Let $v=e^{\tau\varphi(t)}u$.
A computation shows that
\begin{align}
L^+_{\tau} u
&:= e^{- \tau \vp\pr{t}} L^+ v
= e^{- \tau \vp\pr{t}} \pr{\partial_t+\Lambda+\frac{n-2}{2}} e^{\tau\varphi(t)}u \nonumber \\
&= \partial_t u
+ \Lambda u
+  \frac{n-2}{2} u
+ \tau \vp^\prime\pr{t} u \nonumber \\
&=  \partial_t u + \sum_{k \ge 0} \pr{k+ n-2} u_k + \tau \pr{1 + 2 t^{-1}} u,
\label{Lpt}
\end{align}
where the last line follows from an application of \eqref{ord} with $u_k = P_k u$.
It is clear that
\begin{eqnarray}
\norm{t^{-1}e^{-\tau\varphi(t)}L^+v}^2_{L^2(dtd\omega)} &=&
\iint t^{-2}\pr{\partial_t u+\Lambda u+\tau u+2\tau t^{-1} u +\frac{n-2}{2}u }^2 dtd\omega \nonumber \\
&=& \iint t^{-2} \abs{\partial_t u}^2 dtd\om +  2 \iint t^{-2}
\del_t u \sum_{k \ge 0} \pr{k+ n-2} u_k \, dt d\om
\nonumber \\
&+&  2 \tau \iint t^{-2}  \pr{1 + 2 t^{-1}}
u \del_t u \, dt d\om \nonumber\\
&+& \iint t^{-2}\pr{\Lambda u+\tau u+2\tau t^{-1} u +\frac{n-2}{2}
u}^2 dtd\om .
\label{comb1}
\end{eqnarray}
 Keeping in mind that $t < 0$, integration by parts then gives
\begin{eqnarray}
2 \iint t^{-2} \del_t u \sum_{k \ge 0} \pr{k+ n-2} u_k \, dt d\om
&=&  \iint t^{-2} \sum_{k \ge 0} \pr{k+ n-2} \del_t \abs{u_k}^2 \, dt d\om \nonumber \\
&=& - 2 \sum_{k \ge 0} \pr{k+ n-2} \norm{t^{-3/2} u_k}^2_{L^2\pr{ dt
d\om}}, \label{ana1}
\end{eqnarray}
and
\begin{eqnarray}
 2 \tau \int t^{-2} \del_t u \pr{1 + 2 t^{-1}} u
&=& \tau \int t^{-2} \pr{1 + 2 t^{-1}} \del_t\abs{u}^2 dt d\om \nonumber\\
&=& - 2 \tau \norm{ t^{-3/2}u}^2_{L^2\pr{dt d\om}} + 6 \tau \norm{
t^{-2} u}^2_{L^2\pr{d t d\om}} . \label{ana2}
\end{eqnarray}
By the definition of $\Lambda$, we have
\begin{eqnarray}
&&\norm{ t^{-1}\pr{\Lambda +\tau +2\tau t^{-1} +\frac{n-2}{2}} u
}^2_{L^2(dtd\omega)} \nonumber \\ &&=  \sum_{k \ge 0} \iint
t^{-2}\brac{k+ n-2 + \tau \pr{1 + 2 t^{-1}}}^2 \abs{u_k}^2 dt d\om.
\end{eqnarray}
Since
\begin{align*}
\brac{k+n -2 + \tau+2\tau t^{-1}}^2
\geq \frac{3\tau^2}{4} +4(\tau t^{-1})^2 +k(k+n-2) - 2 t^{-1}\pr{k + n - 2} - 2 \tau t^{-1}
\end{align*}
for $|t_0|$ large enough, then
\begin{align*}
&\norm{ t^{-1}\pr{\Lambda +\tau +\tau t^{-1} +\frac{n-2}{2}}
u}^2_{L^2(dtd\omega)} \ge
\frac{3\tau^2}{4}\norm{t^{-1}u}^2_{L^2(dtd\omega)}
+ 4\tau^2 \norm{t^{-2}u}^2_{L^2(dtd\omega)} \\
&+\sum_{k\geq 0}  k(k+n-2) \norm{t^{-1}u_k}^2_{L^2(dtd\omega)}
+ 2 \sum_{k\geq0} \pr{k + n - 2} \norm{t^{-3/2}u_k}^2_{L^2(dtd\omega)}
+ 2 \tau \norm{t^{-3/2}u_k}^2_{L^2(dtd\omega)}.
\end{align*}
Combining the previous estimate with \eqref{comb1}, \eqref{ana1} and \eqref{ana2}, it follows
that
\begin{align*}
\norm{t^{-1} e^{-\tau\varphi(t)}L^+v}^2_{L^2(dtd\omega)}
&\ge \norm{t^{-1} \partial_t u }^2_{L^2(dtd\omega)}
+ \frac{\tau^2}{2} \norm{t^{-1}u}^2_{L^2(dtd\omega)}
+ 2\tau^2 \norm{t^{-2}u}^2_{L^2(dtd\omega)} \\
&+ \sum_{k\geq 0}  k(k+n-2)\norm{t^{-1}u_k}^2_{L^2(dtd\omega)}
+ \frac{\tau^2}{4} \norm{t^{-1}u}^2_{L^2(dtd\omega)} .
\end{align*}
Recalling that $u = e^{- \tau \vp\pr{t}} v$, we see by the triangle inequality that
\begin{align*}
\norm{t^{-1}\del_t u }^2_{L^2(dtd\omega)}
+  \frac{\tau^2}{2} \norm{t^{-1} u}^2_{L^2(dtd\omega)}
+  2\tau^2 \norm{t^{-2}  u }^2_{L^2(dtd\omega)}
&\ge \frac 1 5 \norm{t^{-1}e^{- \tau \vp\pr{t}}\partial_t v }^2_{L^2(dtd\omega)} .
\end{align*}
Substituting this expression into the previous inequality, and using an application of \eqref{lll} gives
\begin{align*}
C \norm{t^{-1} e^{-\tau\varphi(t)}L^+v}^2_{L^2(dtd\omega)}
&\ge \tau^2 \norm{t^{-1}e^{- \tau \vp\pr{t}} v}^2_{L^2(dtd\omega)}
+ \norm{t^{-1}e^{- \tau \vp\pr{t}}\partial_t v }^2_{L^2(dtd\omega)} \\
&+ \sum_{j=1}^n\| t^{-1} e^{- \tau \vp\pr{t}} \Omega_j  v\|^2_{L^2(dtd\omega)},
\end{align*}
since $\Om_j$ acts only on the angular variables.
This implies \eqref{cond}.
\end{proof}

Using a similar process, we also establish an $L^2 - L^2$ Carleman estimate for $L^-$.
Notice that the power on $\tau$ is different here from above.

\begin{lemma}
For every $v \in C^\iny_c\pr{(-\infty, \ t_0)\times S^{n-1}}$, it
holds that
\begin{equation}
\|t^{-{1}} e^{-\tau \varphi(t)} v\|_{L^2(dtd\omega)} \leq C
\tau^{-\frac 1 2} \| e^{-\tau \varphi(t)} L^- v\|_{L^2(dtd\omega)},
\label{key-q22}
\end{equation}
\label{CarL-qq}
where $C$ is a universal constant.
\end{lemma}

\begin{proof}
Recall that $L^-=\partial_t+\frac{n-2}{2}-\Lambda$.
Let $v=e^{\tau\varphi(t)}u$.
Direct computations show that
\begin{align*}
L^-_{\tau} u &:= e^{- \tau \vp\pr{t}} L^- v
= e^{- \tau \vp\pr{t}} \pr{\partial_t+\frac{n-2}{2}-\Lambda} e^{\tau\varphi(t)}u \nonumber \\
&= \partial_t u -\Lambda u +  \frac{n-2}{2} u
+ \tau \vp^\prime\pr{t} u \nonumber \\
&=  \partial_t u -\sum_{k \ge 0} k u_k + \tau \pr{1 + 2 t^{-1}} u,
\end{align*}
where the last line follows from the application of \eqref{ord} with
$u_k = P_k u$.
It is true that
\begin{eqnarray*}
\|e^{-\tau\varphi(t)}L^-v\|^2_{L^2(dtd\omega)} &=&
\iint \brac{\partial_t u+\pr{\frac{n-2}{2}-\Lambda} u+\tau u+2\tau t^{-1} u  }^2 dtd\omega \nonumber \\
&=& \iint  \abs{\partial_t u}^2 dtd\om
+  \iint \brac{\pr{\frac{n-2}{2}-\Lambda} u+\tau u+2\tau t^{-1} u }^2
dtd\om  \\
&+&  2 \iint  \del_t u \sum_{k
\ge 0} (-k) u_k \, dt d\om
+  2 \tau \iint   \pr{1 + 2 t^{-1}}
u \del_t u \, dt d\om .
\end{eqnarray*}
Integration by parts then gives
\begin{eqnarray*}
2 \iint  \del_t u \sum_{k \ge 0} \pr{-k} u_k \, dt d\om =0
\end{eqnarray*}
and
\begin{eqnarray*}
 2 \tau \int  \del_t u \pr{1 + 2 t^{-1}} u
= \tau \int  \pr{1 + 2 t^{-1}} \del_t\abs{u}^2 dt d\om 
= 2 \tau \norm{ t^{-1}u}^2_{L^2\pr{dt d\om}}.
\end{eqnarray*}
Since
\begin{eqnarray*}
&&\|\brac{\pr{\frac{n-2}{2}-\Lambda} u+\tau u+2\tau t^{-1} u
}\|^2_{L^2(dtd\omega)}
\geq 0.
\end{eqnarray*}
then it follows that
\begin{align*}
\norm{ e^{-\tau\varphi(t)}L^-v}^2_{L^2(dtd\omega)} \ge 2 \tau \norm{t^{-1}u}^2_{L^2(dtd\omega)} .
\end{align*}
Recalling that $u = e^{- \tau \vp\pr{t}} v$,
 this implies the estimate \eqref{key-q22}.
\end{proof}

Our next task is to establish $L^p - L^2$ Carleman estimates for the operator $L^-$.
For these results, we require the following lemma which relies on the eigenfunction estimates of Sogge \cite{Sog86}.

\begin{lemma}
Let $N, M \in \N$ and let $\set{c_k}$ be a sequence of numbers such that $\abs{c_k} \le 1$ for all $k$.
For any $v \in L^2\pr{S^{n-1}}$ and all $\frac{2n}{n+2}\leq p\leq 2$, we have
\begin{align}
\|\sum^M_{k=N} c_k P_k v\|_{L^2(S^{n-1})} &\leq C
\brac{M^{\frac{n-2}{2}} \pr{ \sum^M_{k=N}
|c_k|^2}^{\frac{n}{2}}}^{\frac{1}{p}-\frac{1}{2}}\|
v\|_{L^p(S^{n-1})},
\label{haha}
\end{align}
where $C$ depends on $n$ and $p$.
\label{upDown}
\end{lemma}

\begin{proof}
Sogge's \cite{Sog86} eigenfunction estimates state that there is a constant $C$, depending only on $n \ge 3$, such that for any $v\in L^2(S^{n-1})$,
\begin{equation}
\| P_k v\|_{L^{\frac{2n}{n-2}}(S^{n-1})}
\leq Ck^{1-\frac{2}{n}} \|v\|_{L^{\frac{2n}{n+2}}(S^{n-1})}.
\label{sogg}
\end{equation}
Recall that $P_k v = v_k$ is the projection of $v$ onto the space of spherical harmonics of degree $k$.
By orthogonality, H\"older's inequality, and \eqref{sogg},
\begin{eqnarray}
\| P_k v\|^2_{L^{2}(S^{n-1})}
&\leq & \| P_k v\|_{L^{\frac{2n}{n-2}}(S^{n-1})} \|  v\|_{L^{\frac{2n}{n+2}}(S^{n-1})} 
\leq  Ck^{1-\frac{2}{n}}\|  v\|^2_{L^{\frac{2n}{n+2}}(S^{n-1})}.
\label{up}
\end{eqnarray}
It is obvious that
\begin{equation}
 \| P_k v\|_{L^{2}(S^{n-1})} \leq \|v\|_{L^{2}(S^{n-1})}.
 \label{stay}
\end{equation}

Interpolating \eqref{up} and \eqref{stay} gives that
\begin{align}
\| P_k v\|_{L^{2}(S^{n-1})} &\leq C k^{\frac{(n-2)(2-p)}{4p}}
\|v\|_{L^{p}(S^{n-1})} \label{indu}
\end{align}
for all $\frac{2n}{n+2}\leq p\leq 2$.

Now we consider a more general case of the previous inequality.
Let $\{c_k\}$ be a sequence of numbers with $|c_k|\leq 1$.
For all $N\leq M$, by H\"older's inequality, it follows that
\begin{equation*}
\|\sum^{M}_{k=N} c_k P_k v\|^2_{L^2(S^{n-1})}
\leq C\pr{\sum^M_{k=N} |c_k|^2 \|P_k v\|_{L^{\frac{2n}{n-2}}(S^{n-1})}}\|v\|_{L^{\frac{2n}{n+2}}(S^{n-1})}.
\end{equation*}
An application of Sogge's estimate \eqref{sogg} shows that
\begin{equation*}
\|\sum^M_{k=N} c_k P_k v\|_{L^2(S^{n-1})}
\leq C M^\frac{n-2}{2n} \pr{\sum^M_{k=N} |c_k|^2}^{\frac{1}{2}} \|v\|_{L^{\frac{2n}{n+2}}(S^{n-1})}.
\end{equation*}
For any sequence $\set{d_k}$ such that each $|d_k|\leq 1$, it is true that
\begin{equation}
\|\sum^M_{k=N} d_k P_k v\|_{L^2(S^{n-1})}
\leq \| v\|_{L^2(S^{n-1})}.
\label{seqSame}
\end{equation}
As before, we interpolate the last two inequalities (with $d_k = c_k$) and conclude that \eqref{haha} holds.

\end{proof}

Now we prove an $L^p-L^2$ type Carleman estimate for the operator $L^-$.

\begin{lemma}
For every $v \in C^\iny_c\pr{(-\infty, \ t_0)\times S^{n-1}}$ and $\frac{2n}{n+2}< p < 2$,
\begin{equation}
\|t^{-{1}} e^{-\tau \varphi(t)} v\|_{L^2(dtd\omega)} \leq C
\tau^\beta \|t e^{-\tau \varphi(t)} L^- v\|_{L^p(dtd\omega)},
\label{key-}
\end{equation}
where $\be = - \frac 1 2 + \frac{\pr{3n-2}\pr{2-p}}{8p}$ and $C$ depends on $n$ and $p$.
\label{CarL-p}
\end{lemma}

\begin{proof}
To prove this lemma, we introduce the conjugated operator $L^-_\tau$ of $L^-$, defined by
$$L^-_\tau u=e^{-\tau\varphi(t)}L^-(e^{\tau \varphi(t)}u).$$
With $v=e^{\tau \varphi(t)}u$, inequality \eqref{key-} is equivalent to
\begin{equation}
\|t^{-{1}}u \|_{L^2(dtd\omega)}\leq C \tau^\beta \|t L^-_\tau
u\|_{L^p(dtd\omega)}. \label{keyu}
\end{equation}

From \eqref{ord} and \eqref{use}, the operator $L^-_\tau$ takes the form
\begin{equation}
L^-_\tau=\sum_{k\geq 0} (\partial_t+\tau \varphi'(t)-k)P_k.
\label{ord1}
\end{equation}
Since $\disp \sum_{k\geq 0} P_k v= v,$ we split $\disp \sum_{k \ge 0} P_k v$ into two sums.
Let $M=\lceil 2\tau\rceil$ and define
$$ P^+_\tau=\sum_{k> M}P_k, \quad \quad  P^-_\tau=\sum_{k=0}^{M}P_k.      $$
In order to prove the \eqref{keyu}, it suffices to show that
\begin{equation}
\| t^{-1} P^+_\tau {u} \|_{L^2(dtd\omega)} \leq \tau^{\beta}\|
t{L_\tau^- u} \|_{L^{{p}}(dtd\omega)}
\label{key1}
\end{equation}
and
\begin{equation}
\| t^{-1} P^-_\tau {u}\|_{L^2(dtd\omega)}\leq \tau^\beta\| t
{L_\tau^- u} \|_{L^{{p}}(dtd\omega)},
\label{key2}
\end{equation}
for all $u \in C^\iny_c\pr{(-\infty, \ t_0)\times S^{n-1}}$ and $\frac{2n}{n+2}< p < 2$.
The sum of \eqref{key1} and \eqref{key2} will yield \eqref{keyu}, which implies \eqref{key-}.
We first establish \eqref{key1}.
From \eqref{ord1}, we have the first order differential equation
\begin{equation}
P_k L^-_\tau u= (\partial_t+\tau \varphi'(t)-k)P_k u.
\label{sord}
\end{equation}
For $u\in C^\infty_{0}\pr{ (-\infty, \ t_0)\times S^{n-1}}$
, solving the first order differential equation gives that
\begin{equation}
P_k u(t, \omega)
=-\int_{-\infty}^{\infty} H(s-t)e^{k(t-s)+\tau\brac{\varphi(s)-\varphi(t)}} P_k L^-_\tau u (s, \omega)\, ds,
\label{star}
\end{equation}
where $H(z)=1$ if $z\geq 0$ and $H(z)=0$ if $z<0$.

For $k\geq M \ge 2 \tau$, we obtain that
\begin{equation*}
H(s-t)e^{k(t-s)+\tau\brac{\varphi(s)-\varphi(t)}}
\leq e^{-\frac{1}{2}k|t-s|}
\end{equation*}
for all $s, t\in (-\infty, \ t_0)$.
Taking the $L^2\pr{S^{n-1}}$-norm in \eqref{star} gives that
\begin{equation*}
\|P_k u(t, \cdot)\|_{L^2(S^{n-1})}
\leq  \int_{-\infty}^{\infty} e^{-\frac{1}{2}k|t-s|} \|P_k L^-_\tau u(s, \cdot)\|_{L^2(S^{n-1})} \,ds.
\end{equation*}
With the aid of \eqref{indu}, we get
\begin{equation*}
\|P_k u(t, \cdot)\|_{L^2(S^{n-1})}\leq C k^{\frac{(n-2)(2-p)}{4p}}
\int_{-\infty}^{\infty} e^{-\frac{1}{2}k|t-s|} \|L^-_\tau u(s,
\cdot)\|_{L^p(S^{n-1})} \,ds
\end{equation*}
for all $\frac{2n}{n+2}\leq p\leq 2$.
Applying Young's inequality for convolution then yields
\begin{equation*}
\|P_k u\|_{L^2(dt d\omega)}
\leq C k^{\frac{(n-2)(2-p)}{4p}} \pr{\int_{-\infty}^{\infty} e^{-\frac{\sigma}{2}k|z|} dz}^{\frac{1}{\sigma}}\|L^-_\tau u\|_{L^p(dtd\omega)}
\end{equation*}
with $\frac{1}{\sigma}=\frac{3}{2}-\frac{1}{p}$.
A calculation shows that
$$ \pr{\int_{-\infty}^{\infty} e^{-\frac{\sigma}{2}k|z|}}^{\frac{1}{\sigma}}\leq C k^{\frac{1}{p} - \frac{3}{2}}.  $$
Therefore,
\begin{equation*}
\|P_k u\|_{L^2(dt d\omega)}\leq C
k^{-1+\frac{n\pr{2-p}}{4p}} \|L^-_\tau
u\|_{L^p(dtd\omega)}.
\end{equation*}
Squaring and summing up $k> M$ gives that
$$\sum_{k> M} \|P_k u\|^2_{L^2(dt d\omega)}
\leq C \sum_{k> M} k^{-2+\frac{n\pr{2-p}}{2p}} \|L^-_\tau u\|^2_{L^p(dtd\omega)}.$$
Since $\frac{2n}{n+2}<p$, then $-2+\frac{n\pr{2-p}}{2p} <-1$.
Thus, $\disp \sum_{k> M} k^{-2+\frac{n\pr{2-p}}{2p}}$ converges.
Note that at the borderline, where $p=\frac{2n}{n+2}$,  we have that $-2+\frac{n\pr{2-p}}{2p}=-1$ and then the series $\disp \sum_{k\geq M} k^{-2+\frac{n\pr{2-p}}{2p}}$ diverges.
Further calculations show that
$$\sum_{k> M} k^{-2+\frac{n\pr{2-p}}{2p}}\leq C M^{-1+\frac{n\pr{2-p}}{2p}} \leq
C \tau^{-1+\frac{n\pr{2-p}}{2p}}.$$
Recalling that $2\be =- 1 + \frac{\pr{3n-2}\pr{2-p}}{4p}$, since $p < 2$ and $n\geq 3$, then
$-1+\frac{n\pr{2-p}}{2p} \leq 2 \beta$.
Therefore,
\begin{equation*}
\|  P^+_\tau u\|_{L^2(dtd\omega)}\leq C\tau^{\beta}\| L^-_\tau
u\|_{L^p(dtd\omega)},
\end{equation*}
which implies estimate \eqref{key1} since $u$ and $L_\tau^-u$ are supported on $\pr{-\iny, -\abs{t_0}} \times S^{n-1}$, where $\abs{t_0} \ge 1$.

Fix $t\in (-\infty, \ t_0)$ and set $N=\lceil\tau \varphi^\prime(t)\rceil$.
Recall that $\varphi(t)=t+\log t^2$.
An application Taylor's theorem (on a dyadic decomposition of $\frac s t$) shows that for all $s, t \in (-\infty, \ t_0)$
\begin{align}
\varphi(s)-\varphi(t)
&= \varphi'(t)(s-t)+\frac{1}{2}\varphi''(s_0)(s-t)^2
= \varphi'(t)(s-t)-\frac{1}{(s_0)^2}(s-t)^2,
\label{Taylor}
\end{align}
where $s_0$ is some number between $s$ and $t$.
If $s>t$, then
$$S_k(s, t) = e^{k(t-s) + \tau\brac{\vp\pr{s - \vp\pr{t}}}} \leq e^{-(k-\tau\varphi'(t))(s-t)-\frac{\tau}{t^2}(s-t)^2},$$
so that
\begin{equation}
H(s-t) S_k(s, t)\leq e^{-|k- N 
||s-t|-\frac{\tau}{t^2}(s-t)^2}.
\label{case1}
\end{equation}
First we consider the case $N\leq k\leq M$.
From \eqref{star}, we sum over $k$ to get
\begin{equation}
\| \sum^M_{k=N} P_k u(t, \cdot)\|_{L^2(S^{n-1})}
\leq \int_{-\infty}^{\infty} \| \sum^M_{k=N} H(s-t)S_k(s,t) P_k L^-_\tau u(s, \cdot) \|_{L^2(S^{n-1})}\, ds.
\label{back}
\end{equation}
With $c_k= H(s-t)S_k(s,t)$, it is clear that $|c_k|\leq 1$.
Therefore, Lemma \ref{upDown} is applicable, so we may apply estimate \eqref{haha} to obtain
\begin{eqnarray}
&&\|\sum^M_{k=N}H(s-t)S_k(s,t) P_k L^-_\tau u(s, \cdot)\|_{L^2(S^{n-1})}\leq \nonumber \medskip\\
&& C \brac{ \tau^{\frac{n-2}{{2}}}\pr{\sum^M_{k=N}
H(s-t)|S_k(s,t)|^2}^{\frac{n}{2}}}^{\frac{1}{p}-\frac{1}{2}}
 \|L^-_\tau u(s, \cdot) \|_{L^p(S^{n-1})}
\label{mixSum}
\end{eqnarray}
for all $\frac{2n}{n+2}< p < 2$. Now we use the inequality
\eqref{case1} to bound $\disp \sum^M_{k=N} H(s-t)|S_k(s,t)|^2$.
\begin{align}
\sum^M_{k=N} H(s-t)|S_k(s,t)|^2
&\leq \pr{\sum^M_{k=N+1} e^{-2|k- N ||s-t|}+1}e^{ -\frac{2\tau}{t^2}(s-t)^2}
\nonumber \\
&\leq  C \pr{ \frac{1}{|s-t|}+1} e^{ -\frac{\tau}{t^2}(s-t)^2} .
\label{sumBnd}
\end{align}
Therefore, from \eqref{mixSum},
\begin{eqnarray*}
&&\|\sum^M_{k=N}H(s-t)S_k(s,t) P_k L^-_\tau u(s,
\cdot)\|_{L^2(S^{n-1})} \nonumber \medskip\\ &&
\leq  C \tau^{\alpha_1} (|s-t|^{-\alpha_2}+1)e^{-\frac{\alpha_2\tau}{t^2}(s-t)^2}\| L^-_\tau u(s, \cdot) \|_{L^p(S^{n-1})},
\end{eqnarray*}
where $\alpha_1=\frac{(n-2)(2-p)}{4p}$ and $\alpha_2=\frac{n(2-p)}{4p}$.
We see that
$$ e^{-\frac{\alpha_2\tau}{t^2}(s-t)^2}\leq C_j \pr{1+\frac{\alpha_2\tau}{t^2}(s-t)^2}^{-j}   $$
for all $j\geq 0$ so that with $j=\frac{1}{2}$, we have
\begin{equation}
e^{-\frac{\alpha_2\tau}{t^2}(s-t)^2}\leq C|t|\pr{1+\sqrt{\tau}|s-t|}^{-1}.
\label{expBnd}
\end{equation}
Thus,
$$ \|\sum^M_{k=N}H(s-t)S_k(s,t) P_k L^-_\tau u(s,\cdot)\|_{L^2(S^{n-1})}
\leq \frac{C\tau^{\alpha_1}|t|(|s-t|^{-\alpha_2}+1)\| L^-_\tau u(s,
\cdot) \|_{L^p(S^{n-1})}}{1+\sqrt{\tau}|s-t|}.
$$
It follows from \eqref{back} that
\begin{equation}
|t|^{-1}\| \sum^M_{k=N} P_k u(t, \cdot)\|_{L^2(S^{n-1})} \leq C
\tau^{\alpha_1} \int_{-\infty}^{\infty}\frac{(|s-t|^{-\alpha_2}+1)\|
L^-_\tau u(s, \cdot) \|_{L^p(S^{n-1})}}{1+\sqrt{\tau}|s-t|}.
\label{mile}
\end{equation}

For $k\leq N-1$, we solve the first order differential equation
\eqref{sord} as
\begin{equation}
P_k u(t, \omega)=\int_{-\infty}^{\infty} H(t-s)S_k(s, t) P_k L^-_\tau u\pr{s, \om}\, ds
\label{star1}.
\end{equation}
It follows from \eqref{Taylor} that for any $s, t$
\begin{equation}
H(t-s) S_k(s, t)\leq e^{-|N- 1 - k  
||s-t|-\frac{\tau}{s^2}(t-s)^2}.
\label{case2}
\end{equation}
Arguing as before, we use the upper bound for $H(t-s)S_k(s,t)$ in \eqref{case2}  to similarly conclude that
\begin{equation}
\| \sum_{k=0}^{N-1} P_k u(t, \cdot)\|_{L^2(S^{n-1})} \leq
C\tau^{\alpha_1} \int_{-\infty}^{\infty}
\frac{|s|(|s-t|^{-\alpha_2}+1)\| L^-_\tau u(s, \cdot)
\|_{L^p(S^{n-1})}}{1+\sqrt{\tau}|s-t|}.
\label{mile1}
\end{equation}
Since $s,t $ are in $(-\infty, \ t_0)$ with $|t_0|$ large enough, we combine estimates \eqref{mile} and \eqref{mile1} to arrive at
\begin{equation*}
|t|^{-1}\| P_\tau^- u(t, \cdot)\|_{L^2(S^{n-1})} \leq C
\tau^{\alpha_1} \int_{-\infty}^{\infty}
\frac{|s|(|s-t|^{-\alpha_2}+1)\| L^-_\tau u(s, \cdot)
\|_{L^p(S^{n-1})}}{1+\sqrt{\tau}|s-t|}.
\end{equation*}
Applying Young's inequality for convolution, we get
\begin{equation*}
\| t^{-1} P_\tau^- u (t, \cdot) \|_{L^2(dtd\omega)}
\leq C\tau^{\alpha_1} \brac{\int_{-\infty}^{\infty} \pr{\frac{ |z|^{-\alpha_2}+1} {1+\sqrt{\tau}|z|}}^\si \, dz}^{\frac{1}{\si}} \| tL^-_\tau u \|_{L^p(dtd\omega)},
\end{equation*}
where $\frac{1}{\si}=\frac{3}{2}-\frac{1}{p}$.
Since $\al_2 \in \pr{0, \frac 1 2}$ and $\si \in \pr{1, 1 + \frac 1 {n-1}}$ for our
range of $p$, then a direct calculation shows that
$$\brac{\int_{-\infty}^{\infty} \pr{\frac{ |z|^{-\alpha_2}+1} {1+\sqrt{\tau}|z|}}^\si \, dz}^{\frac{1}{\si}} \leq C\tau^{-\frac{1}{2\si}+\frac{\alpha_2}{2}}.
$$
Therefore,
\begin{equation*}
\| t^{-1} P_\tau^- u (t, \cdot) \|_{L^2(dtd\omega)}\leq
C\tau^{-\frac{1}{2\si}+\frac{\alpha_2}{2}+\alpha_1} \| t L^-_\tau
u \|_{L^p(dtd\omega)}.
\end{equation*}
Since
$-\frac{1}{2\si}+\frac{\alpha_2}{2}+\alpha_1=\beta$,
this completes \eqref{key2}, and the lemma is proved.
\end{proof}

We now have all of the ingredients needed to prove the general $L^p - L^q$ Carleman estimate for the Laplace operator given in Theorem \ref{Carlpq}.
The lemmas that we have established within this section were adapted from the ideas in \cite{Reg99} and \cite{Jer86}.
Note that in \cite{Reg99}, a version of the following theorem is proved for $p=\frac{14n-4}{7n+10}$ and $q = 2$.
Similar $L^p - L^2$ Carleman estimates have been shown in  \cite{BKRS88}.

\begin{proof}[Proof of Theorem \ref{Carlpq}]
Let $u\in C^{\infty}_{0}\pr{B_{R_0}(x_0)\backslash\set{x_0} }$.
We first consider the case of $p \in \pr{\frac{2n}{n-2}, 2}$.
An application of \eqref{cond} from Lemma \ref{Car22} applied to $v$, followed by an
application of \eqref{key-} from Lemma \ref{CarL-p} applied to $L^+v$ shows that
\begin{align*}
\tau \norm{t^{-1} e^{-\tau \varphi(t)}v}_{L^2(dtd\omega )}
&+\norm{t^{-1}e^{-\tau \varphi(t)} \partial_t v}_{L^2(dtd\omega )}
+\sum_{j=1}^n \norm{t^{-1}e^{-\tau \varphi(t)} \Omega_j v }_{L^2(dtd\omega )}   \\
&\leq C\norm{t^{-1} e^{-\tau \varphi(t)} L^+v}_{L^2(dtd\omega )} \\
&\le C\tau^{-\frac 1 2 + \frac{\pr{3n-2}\pr{2-p}}{8p}}\norm{t
e^{-\tau \varphi(t)} L^- L^+v}_{L^p(dtd\omega )}.
\end{align*}
Now consider when $p = 2$.
Lemma \ref{Car22} combined with \eqref{key-q22} from Lemma \ref{CarL-qq} implies that
\begin{align*}
\tau \norm{t^{-1} e^{-\tau \varphi(t)}v}_{L^2(dtd\omega )}
&+\norm{t^{-1}e^{-\tau \varphi(t)} \partial_t v}_{L^2(dtd\omega )}
+\sum_{j=1}^n \norm{t^{-1}e^{-\tau \varphi(t)} \Omega_j v }_{L^2(dtd\omega )}   \\
&\leq C\norm{t^{-1} e^{-\tau \varphi(t)} L^+v}_{L^2(dtd\omega )} \\
&\le C\tau^{-\frac 1 2}\norm{ e^{-\tau \varphi(t)} L^-
L^+v}_{L^2(dtd\omega )} \\ &\le C\tau^{-\frac 1 2}\norm{t e^{-\tau
\varphi(t)} L^- L^+v}_{L^2(dtd\omega )}.
\end{align*}
From the last two inequalities, recalling the definitions of $t$, $\vp$, and $L^\pm$, we get that for any $p \in \pb{\frac{2n}{n+2}, 2}$,
\begin{align}
&\tau \|(\log r)^{-1} e^{-\tau \phi(r)}u\|_{L^2(r^{-n}dx)}
+ \|(\log r )^{-1} e^{-\tau \phi(r)}r \nabla u\|_{L^2(r^{-n}dx)}
\nonumber \medskip\\
&\leq  C \tau^{\beta}  \|(\log r ) e^{-\tau \phi(r)} r^2 \LP u\|_{L^p(r^{-n}dx)},
\label{dodo}
\end{align}
where we have set $\be = - \frac 1 2 + \frac{\pr{3n-2}\pr{2-p}}{8p}$.

Notice that by the triangle inequality
\begin{align}
\|\nabla [(\log r)^{-1} e^{-\tau \phi(r)}  u r^\frac{-n+2}{2}]\|_{L^2}
&\leq \| (\log r)^{-1} e^{-\tau \phi(r)} u
\|_{L^2(r^{-n}dx)}+ \| (\log r)^{-2} e^{-\tau \phi(r)} u
\|_{L^2(r^{-n}dx)}  \nonumber \medskip\\
&+\tau  \| (\log r)^{-1} e^{-\tau \phi(r)}\phi'(r) r u
\|_{L^2(r^{-n}dx)} \nonumber \medskip\\ &+\frac{n}{2}\|(\log r)^{-1}
e^{-\tau \phi(r)}  u \|_{L^2(r^{-n}dx)}
+\|(\log r)^{-1} e^{-\tau \phi(r)}  r \nabla u \|_{L^2(r^{-n}dx)} \nonumber \medskip \\
&\leq C\tau \| (\log r)^{-1} e^{-\tau \phi(r)} u \|_{L^2(r^{-n}dx)}
\nonumber \medskip\\  &+C \| (\log r)^{-1} e^{-\tau \phi(r)} r
\nabla u \|_{L^2(r^{-n}dx)}, \label{dodo1}
\end{align}
where we have used that $\phi'(r)=\frac{1}{r}+\frac{2}{r\log r} \le \frac 1 r$ since $r \le R_0 \le 1$.
By the Sobolev embedding theorem,
\begin{align}
 \|(\log r)^{-1} e^{-\tau \phi(r)} u\|_{L^\frac{2n}{n-2}(r^{-n}dx)}
&= \ \|(\log r)^{-1} e^{-\tau \phi(r)} u r^\frac{-n+2}{2}\|_{L^\frac{2n}{n-2}\pr{B_{R_0}}} \nonumber \\
&\le c \|\nabla [(\log r)^{-1} e^{-\tau \phi(r)}  u r^\frac{-n+2}{2}]\|_{L^2\pr{B_{R_0}}} \nonumber \\
&\le C\tau \| (\log r)^{-1} e^{-\tau \phi(r)} u \|_{L^2(r^{-n}dx)}
\nonumber \\ &+C \| (\log r)^{-1} e^{-\tau \phi(r)} r \nabla u \|_{L^2(r^{-n}dx)}  \nonumber \\
&\leq  C \tau^\be \|(\log r ) e^{-\tau \phi(r)} r^2 \LP
u\|_{L^p(r^{-n} dx)}, \label{L2*Est}
\end{align}
where the last two inequalities are due to \eqref{dodo1} and \eqref{dodo}, respectively.
From (\ref{dodo}), it is clear that
\begin{equation}
\|(\log r)^{-1} e^{-\tau \phi(r)}u\|_{L^2(r^{-n}dx)}
\leq  C \tau^{\be -1} \|(\log r ) e^{-\tau \phi(r)} r^2 \LP u\|_{L^p(r^{-n} dx)}.
\label{L2Est}
\end{equation}
We are going to do a interpolation with the last two inequalities.
Choose $\la \in \pr{0,1}$ so that $q = 2 \la + \pr{1-\la} \frac{2n}{n-2}$.
By H\"older's inequality,
\begin{align*}
\|(\log r)^{-1} e^{-\tau \phi(r)}u\|_{L^q(r^{-n}dx)}
&\le \|(\log r)^{-1} e^{-\tau \phi(r)}u\|_{L^2(r^{-n}dx)}^{\frac{2\la}{q}} \|(\log r)^{-1} e^{-\tau \phi(r)}u\|_{L^{\frac{2n}{n-2}}(r^{-n}dx)}^{\frac{2n\pr{1-\la}}{\pr{n-2}q}}.
\end{align*}
Since $\la  = \frac{2q - n\pr{q - 2} }{4}$, if we set $\te = \frac{2\la}{q} = \frac{2q  - n\pr{q -2}}{2q}$, then $1 - \te = \frac{n\pr{q-2}}{2q} = \frac{2n\pr{1-\la}}{\pr{n-2}q}$ and we have that $0\leq \theta\leq 1$.
Therefore,
\begin{align*}
& \|(\log r)^{-1} e^{-\tau \phi(r)}u\|_{L^q(r^{-n}dx)} \\
&\leq  \|(\log r)^{-1} e^{-\tau \phi(r)}u\|_{L^{2}(r^{-n}dx)}^{\theta} \|(\log
r)^{-1} e^{-\tau\phi(r)}u\|_{L^\frac{2n}{n-2}(r^{-n}dx)}^{1-\theta} \\
&\le \brac{C\tau^{\beta-1} \|(\log r ) e^{-\tau \phi(r)} r^2 \LP u\|_{L^p(r^{-n} dx)}}^\te
\brac{C \tau^\beta  \|(\log r ) e^{-\tau \phi(r)} r^2 \LP u\|_{L^p(r^{-n} dx)}}^{1 - \te} \\
&= C \tau^{\beta-\te} \|(\log r ) e^{-\tau \phi(r)} r^2 \LP u\|_{L^p(r^{-n} dx)},
\end{align*}
where the last inequality follows from \eqref{L2Est} and \eqref{L2*Est}.
That is, for any $2\leq q\leq \frac{2n}{n-2}$,
\begin{equation*}
\tau^{\frac 3 2 - \frac{\pr{3n-2}\pr{2-p}}{8p} - \frac{n\pr{q -2}}{2q} }\|(\log r)^{-1} e^{-\tau
\phi(r)}u\|_{L^q(r^{-n}dx)} \leq C \|(\log r ) e^{-\tau \phi(r)} r^2
\LP u\|_{L^p(r^{-n} dx)}.
\end{equation*}
Since \eqref{dodo} implies that
$$\tau^{\frac 1 2 - \frac{\pr{3n-2}\pr{2-p}}{8p}}   \|(\log r )^{-1} e^{-\tau \phi(r)}r \nabla u\|_{L^2(r^{-n}dx)} \leq  C \|(\log r ) e^{-\tau \phi(r)} r^2 \LP u\|_{L^p(r^{-n}dx)},$$
adding the previous two inequalities gives the proof of Theorem \ref{Carlpq}.
\end{proof}

We prove a quantitative Caccioppoli inequality for the second order elliptic equation (\ref{goal}) with singular lower order terms.
This Caccioppoli inequality is known, but since we want to show how the estimate depends on the norms of $W$ and $V$, we present the details of the proof.

\begin{lemma}
Assume that for some $s \in \pb{n, \iny}$ and $t \in \pb{ \frac n 2, \iny}$, $\norm{W}_{L^s\pr{B_{R}}} \le K$ and $\norm{V}_{L^t\pr{B_{R}}} \le M$.
Let $u$ be a solution to \eqref{goal} in $B_R$.
Then there exists a constant $C$, depending only on $n$, $s$ and $t$, such that
\begin{equation}
\|\nabla u\|^2_{L^2(B_r)}\leq
C\brac{\frac{1}{(R-r)^2}+K^{\frac{2s}{s-n}}+M^{\frac{2t}{2t-n}}}\|
u\|^2_{L^2(B_R)}
\end{equation}
for any $r<R$.
\label{CaccLem}
\end{lemma}

\begin{proof}
We need to decompose $W$ and $V$.
Let
$$ W(x)=\overline{W}_{K_0}+W_{K_0},   \quad  V(x)=\overline{V}_{M_0}+V_{M_0},        $$
where $${\overline{W}}_{K_0}=W(x)\chi_{\set{|W(x)|\leq \sqrt{K_0}}}, \quad
{W}_{K_0}=W(x)\chi_{\set{|W(x)|>\sqrt{K_0}}},  $$ and
$${\overline{V}}_{M_0}=V(x)\chi_{\set{|V(x)|\leq \sqrt{M_0}}}, \quad
{V}_{M_0}=V(x)\chi_{\set{|V(x)|>\sqrt{M_0}}}, \quad  $$
for some $K_0, M_0$ to be determined.

For any $q$ with $1\leq q\leq s$, we have that
\begin{equation}
\|W_{K_0}\|_{L^q}\leq
K_0^{-\frac{s-q}{2q}}\|W_{K_0}\|_{L^{s}}^{\frac{s}{q}}\leq
K_0^{-\frac{s-q}{2q}}\|W\|_{L^{s}}^{\frac{s}{q}} \leq
K_0^{-\frac{s-q}{2q}}K^{\frac{s}{q}}.
\label{mmo}
\end{equation}
Similarly, for any $q$ with $1\leq q\leq t$, it holds
that
\begin{equation}
\|V_{M_0}\|_{L^q}\leq
M_0^{-\frac{t-q}{2q}}\|V_{M_0}\|_{L^{t}}^{\frac{t}{q}}\leq
M_0^{-\frac{t-q}{2q}}\|V\|_{L^{t}}^{\frac{t}{q}} \leq
M_0^{-\frac{t-q}{2q}}M^{\frac{t}{q}}.
\label{who}
\end{equation}

Let $\eta$ be a smooth cut-off function such that $\eta(x) \equiv 1$ in $B_r$,
$\eta(x) \equiv 0$ outside $B_R$ with $B_R\subset  B_1$. Then $|\nabla \eta|\leq \frac{C}{|R-r|}$.
Multiplying both sides of equation (\ref{goal}) by $\eta^2 u$ and integrating by parts, we obtain
\begin{equation}
\int V\eta^2 u^2 \, dx+\int W\cdot \nabla u \, \eta^2 u \, dx-2\int \nabla u\cdot\nabla \eta \, \eta \, u \, dx
=\int |\nabla u|^2 \eta^2 \,
dx.
\label{cacc}
\end{equation}
We estimate the terms on the left side of \eqref{cacc}. To
control $\int V\eta^2 u^2 \, dx$, we have
\begin{equation*}
\int V \eta^2 u^2\, dx \leq \int |{\overline{V}_{M_0}}|\eta^2 u^2 \,
dx+ \int |{{V}_{M_0}}|\eta^2 u^2 \, dx.
\end{equation*}
It is clear that
\begin{equation}
\int |{\overline{V}_{M_0}}|\eta^2 u^2 \, dx\leq
M_0^{\frac{1}{2}}\int \eta^2 u^2 \, dx.
\label{sss}
\end{equation}
By H\"older's inequality, \eqref{who} with $q=\frac{n}{2}$, and Sobolev imbedding, we get
\begin{eqnarray}
\abs{\int {V}_{M_0}\eta^2 u^2 \, dx }
&\leq&\pr{\int|{{V}_{M_0}}|^{\frac{n}{2}}\, dx}^{\frac{2}{n}}
\pr{\int \abs{\eta^2u^2}^{\frac{n}{n-2}}\, dx}^{\frac{n-2}{n}} \nonumber \medskip\\
&\leq & c_n {M_0}^{-\frac{2t-n}{2n}}M^{\frac{2t}{n}}\int |\nabla
(\eta u)|^2 \, dx. \label{mmm}
\end{eqnarray}
Taking $c_n {M_0}^{-\frac{2t-n}{2n}}M^{\frac{2t}{n}}=\frac 1 {16}$, i.e
$M_0=(16 c_n M^{\frac{2t}{n}})^{\frac{2n}{2t-n}}$, from
(\ref{sss}) and (\ref{mmm}), we get
\begin{eqnarray}
\int V \eta^2 u^2\, dx & \leq & C M^{\frac{2t}{2t-n}}\int
|\eta u|^2 \, dx+\frac{1}{16}\int |\nabla (\eta u)|^2 \, dx
\nonumber
\medskip \\
& \leq &C M^{\frac{2t}{2t-n}}\int |\eta u|^2 \,
dx+\frac{1}{8}\int |\nabla \eta|^2 u^2 \, dx+\frac{1}{8}\int |\nabla
u|^2 \eta^2 \, dx.
\label{more1}
\end{eqnarray}

We estimate the second term in the lefthand side of (\ref{cacc}).
It is true that
\begin{equation}
\int W \cdot \nabla u \eta^2 u\, dx
= \int {\overline{W}_{K_0}} \cdot \nabla u \eta^2 u\, dx
+ \int {{W}_{K_0}} \cdot \nabla u \eta^2 u\, dx
\label{w1}
\end{equation}
By Young's inequality, we have
\begin{equation}
\int |{\overline{W}_{K_0}}|  |\nabla u| \eta^2 u\, dx\leq \frac{1}{16}
\int |\nabla u|^2\eta^2\,dx +CK_0 \int \eta^2 u^2 \, dx.
\label{w2}
\end{equation}
By H\"older's inequality, \eqref{mmo} with $q=n$, and Sobolev imbedding, we get
\begin{eqnarray}
\int {W}_{K_0} \cdot \nabla u \eta^2 u \, dx.
&\leq & \pr{\int |{{W}_{K_0}}|^{{n}}\, dx}^{\frac{1}{n}} \pr{\int |\nabla u\cdot \eta
|^{2}\, dx}^{\frac{1}{2}} \pr{\int |u \eta |^{\frac{2n}{n-2}}\,
dx}^{\frac{n-2}{2n}}
\nonumber \medskip \\
& \leq & K_0^{-\frac{s-n}{2n}}K^{\frac{s}{n}}\||\nabla
u|\eta\|_{L^2}
\|\nabla(\eta u)\|_{L^2} \nonumber \medskip \\
&\leq & K_0^{-\frac{s-n}{2n}}K^{\frac{s}{n}}\||\nabla
u|\eta|\|^2_{L^2}+K_0^{-\frac{s-n}{2n}}K^{\frac{s}{n}}
\|\nabla(\eta u)\|^2_{L^2}.
\label{w3}
\end{eqnarray}
We choose $K_0^{-\frac{s-n}{2n}}K^{\frac{s}{n}}=\frac 1 {16}$, that is,
$K_0=16^{\frac{2n}{s-n}}K^{\frac{2s}{s-n}}$. Combining
(\ref{w1}), (\ref{w2}) and (\ref{w3}) gives that
\begin{eqnarray}
\int W \cdot \nabla u \, \eta^2 u\, dx &\leq& \frac{1}{8} \||\nabla u|
\eta\|^2_{L^2}+CK^{\frac{2s}{s-n}}
\|u\eta\|^2_{L^2}+\frac{1}{16}\|\nabla (\eta
u)\|^2_{L^2} \nonumber \\
&\leq & \frac{1}{4}\||\nabla u|
\eta\|^2_{L^2}+CK^{\frac{2s}{s-n}}
\|u\eta\|^2_{L^2}+\frac{1}{8}\||\nabla \eta| u \|^2_{L^2}.
\label{more2}
\end{eqnarray}
Finally, note that
\begin{align*}
2\int \nabla u\cdot\nabla \eta \, \eta \, u \, dx
&\le \frac 1 8 \int |\nabla u|^2 \eta^2 \, dx
+ 8 \int |\nabla \eta|^2 \abs{u}^2 \, dx
\end{align*}

Together with (\ref{cacc}), (\ref{more1}) and (\ref{more2}), we
obtain
\begin{equation*}
\int |\nabla u|^2 \eta^2
\leq C \pr{M^{\frac{2t}{2t-n}} +CK^{\frac{2s}{s-n}}} \int |u\eta|^2\,dx+ C \int
|\nabla \eta|^2u^2\,dx.
\end{equation*}
By the assumptions on $\eta$, we arrive at the conclusion in the
lemma.
\end{proof}

\section{Vanishing order}
\label{vanOrd}

Using the Carleman estimate in Theorem \ref{CarlpqVW}, we establish a three-ball inequality that serves as the main tool in the proof of Theorem \ref{thh}.
We consider solutions to \eqref{goal} with first order term $W$ and zeroth order term $V$.
The arguments we present are similar to those that appear in \cite{Ken07}.
By the translation invariance of the equations, the following three-ball inequality holds for balls centered at any $x_0 \in \R^n$ .
Without loss of generality, we assume that $x_0$ is the origin.

\begin{lemma}
Let $0 < r_0< r_1< R_1 < R_0$, where $R_0 < 1$ is sufficiently small.
Let $s \in \pb{ \frac{3n-2}{2}, \iny}$, $t \in \pb{ n\pr{\frac{3n-2}{5n-2}}, \iny}$.
Assume that for some $K, M \ge 1$, $\|W\|_{L^s\pr{B_{R_0}}} \le K$ and $\|V\|_{L^t\pr{B_{R_0}}} \le M$.
Let $u : B_{R_0} \to \C$ be a solution to \eqref{goal} in $B_{R_0}$.
Then there exists a constant $C$, depending on $n$, $s$ and $t$, such that
\begin{align}
\|u\|_{L^\infty \pr{B_{3r_1/4}}} &\le C F\pr{r_1}^{\frac n 2}  |\log
r_1| \brac{ (K+|\log r_0|)F\pr{r_0} \|u\|_{L^\iny(B_{2r_0})}}^{k_0} \nonumber \\
&\times \brac{(K+|\log
R_1|)  F\pr{R_1}\|u\|_{L^\iny(B_{R_1})}}^{1 - k_0} \nonumber \\
&+C F\pr{r_1}^{\frac n 2} \pr{\frac{R_1 }{r_1}}^{\frac n 2} \pr{1 +\frac{|\log r_0|}{K}} \nonumber \\
&\times \exp\brac{\pr{C_1 K^\kappa + C_2 M^\mu} \pr{\phi\pr{\frac{R_1}{2}}-\phi(r_0)}} \|u\|_{L^\iny(B_{2r_0})},
\label{three}
\end{align}
where $\disp k_0 = \frac{\phi(\frac{R_1}{2})-\phi(r_1)}{\phi(\frac{R_1}{2})-\phi(r_0)}$, $F\pr{r} = 1 + r K^{\frac{s}{s-n}} + r M^{\frac{t}{2t-n}}$, and $\kappa$, $\mu$, $C_1$ and $C_2$ are as in Theorem \ref{CarlpqVW}.
\label{threeBalls}
\end{lemma}

\begin{proof}
Let $r_0< r_1< R_1$.
Choose a smooth function $\eta\in C^\infty_{0}(B_{R_0})$ with $B_{2R_1}\subset B_{R_0}$.
We use the notation $\brac{a,b}$ to denote a closed annulus with inner radius $a$ and outer radius $b$.
Let
$$D_1=\brac{\frac{3}{2}r_0, \frac{1}{2}R_1 }, \quad  \quad
D_2= \brac{r_0, \frac{3}{2}r_0}, \quad \quad
D_3=\brac{\frac{1}{2}R_1, \frac{3 }{4}R_1}.$$ Let $\eta=1$ on $D_1$
and $\eta=0$ on $[0, \ r_0]\cup \brac{\frac{3}{4}R_1, \ R_1}$. We
have $|\nabla \eta|\leq \frac{C}{r_0}$ and $|\nabla^2\eta|\leq
\frac{C}{r_0^2}$ on $D_2$. Similarly, $|\nabla \eta|\leq
\frac{C}{R_1}$ and $|\nabla^2 \eta|\leq\frac{C}{R_1^2}$ on $D_3$.

Since $u$ is a solution to \eqref{goal} in $B_{R_0}$, then, as per the discussion before the statement of Theorem \ref{thh}, $u \in L^\iny\pr{B_{R_1}} \cap W^{1,2}\pr{B_{R_1}} \cap W^{2,p}\pr{B_{R_1}}$.
Therefore, by regularization, the estimate in Theorem \ref{CarlpqVW} holds for $\eta u$.
Substituting $\eta u$ into the Carleman estimates in Theorem \ref{CarlpqVW} and using that $u$ is a solution to equation \eqref{goal}, we get
\begin{align*}
\tau^{\be_0} \|(\log r)^{-1} e^{-\tau \phi(r)} \eta u\|_{L^2(r^{-n}dx)}
&\leq  C_0 \|(\log r ) e^{-\tau \phi(r)} r^2\pr{ \LP \pr{\eta u} + W \cdot \gr\pr{\eta u} + V \eta u}\|_{L^p(r^{-n} dx)} \\
&=  C_0 \|(\log r ) e^{-\tau \phi(r)} r^2\pr{ \LP \eta \, u + 2 \gr \eta \cdot \gr u + W \cdot \gr \eta \, u }\|_{L^p(r^{-n} dx)},
\end{align*}
whenever
$$\tau \ge C_1 K^{\kappa} + C_2 M^{\mu}.$$

Then
\begin{equation}
\tau^{\beta_0} \|(\log r)^{-1} e^{-\tau \phi(r)} u\|_{L^2(D_1, r^{-n}dx )}\leq J,
\label{jjj}
\end{equation}
where
$$ J= C_0 \|(\log r) e^{-\tau \phi(r)} r^2 (\LP \eta \, u+W\cdot \nabla \eta \, u + 2\nabla \eta \cdot \nabla u)\|_{L^p(D_2\cup D_3, r^{-n} dx)}.$$
An application of H\"older's inequality shows that
\begin{align*}
\|(\log r) e^{-\tau \phi(r)} r^2 \LP \eta \, u\|_{L^p(D_2\cup D_3, r^{-n} dx)}
&\le \| \pr{\log r} r^{2} \LP \eta\|_{L^\iny\pr{D_2}} \| r^{- \frac n p}\|_{L^{\frac{2p}{2-p}}\pr{D_2}}\| e^{-\tau \phi(r)} u\|_{L^2(D_2)} \\
&+\| \pr{\log r} r^{2} \LP \eta\|_{L^\iny\pr{D_3}} \| r^{- \frac n p}\|_{L^{\frac{2p}{2-p}}\pr{D_3}}\| e^{-\tau \phi(r)} u\|_{L^2(D_3)}
\end{align*}
and
\begin{align*}
\|(\log r) e^{-\tau \phi(r)} r^2 \gr \eta \cdot \gr u\|_{L^p(D_2\cup D_3, r^{-n} dx)}
&\le \| \pr{\log r} r^{2} \gr \eta\|_{L^\iny\pr{D_2}} \| r^{- \frac n p}\|_{L^{\frac{2p}{2-p}}\pr{D_2}}\| e^{-\tau \phi(r)} \gr u\|_{L^2(D_2)} \\
&+\| \pr{\log r} r^{2} \gr \eta\|_{L^\iny\pr{D_3}} \| r^{- \frac n p}\|_{L^{\frac{2p}{2-p}}\pr{D_3}}\| e^{-\tau \phi(r)} \gr u\|_{L^2(D_3)}.
\end{align*}
As in the proof of Theorem \ref{CarlpqVW} (see the computations in \eqref{hod2}), since $\frac{2p}{2-p} \le s$, then
\begin{align*}
&\|(\log r) e^{-\tau\phi(r)} r^2 W\cdot\nabla \eta \, u\|_{L^p(r^{-n} dx, D_2 \cup D_3)} \\
&\le c \|W\|_{L^{s}\pr{D_2}} \norm{\gr \eta}_{L^\iny\pr{D_2}}\| e^{-\tau \phi(r)} r u \|_{L^2(D_2, r^{-n} dx)} \\
&+ c \|W\|_{L^{s}\pr{D_3}} \norm{\gr \eta}_{L^\iny\pr{D_3}}\| e^{-\tau \phi(r)} r u \|_{L^2(D_3, r^{-n} dx)} \\
&\le cK r_0^{-\frac{n}{2}}\|e^{-\tau \phi(r)} u\|_{{L^2(D_2)}}
+ cK R_1^{-\frac{n}{2}}\|e^{-\tau \phi(r)} u \|_{L^2(D_3)},
\end{align*}
where we have used the bounds on $\abs{\gr \eta}$.
By the bounds of $\eta$ in $D_2$ and $D_3$, we obtain
\begin{eqnarray*}
J &\leq & C |\log r_0| r_0^{-\frac{n}{2}}\pr{\|e^{-\tau \phi(r)}
u\|_{{L^2(D_2)}}+ r_0 \|e^{-\tau \phi(r)} \nabla
u\|_{{L^2(D_2)}}}
+CK r_0^{-\frac{n}{2}}\|e^{-\tau \phi(r)}u\|_{{L^2(D_2)}} \nonumber \medskip \\
&+& C|\log R_1|R_1^{-\frac{n}{2}}\pr{\|e^{-\tau \phi(r)} u\|_{{L^2(D_3)}}+ R_1
\|e^{-\tau \phi(r)} \nabla u\|_{{L^2(D_3)}}}
+CK R_1^{-\frac{n}{2}}\|e^{-\tau \phi(r)} u\|_{{L^2(D_3)}}.
\end{eqnarray*}
It follows that
\begin{eqnarray*}
J &\leq & C\pr{K + |\log r_0|} r_0^{-\frac{n}{2}} e^{-\tau \phi(r_0)}  \pr{\| u\|_{{L^2(D_2)}}+ r_0
\|\nabla u\|_{{L^2(D_2)}}} \nonumber \medskip \\
&+& C\pr{K+ |\log R_1|} R_1^{-\frac{n}{2}}e^{-\tau \phi\pr{\frac{R_1}{2}}} \pr{\|
u\|_{{L^2(D_3)}}+ R_1 \| \nabla u\|_{{L^2(D_3)}}},
\end{eqnarray*}
where we have used that $e^{-\tau\phi(r)}$ is a decreasing function with respect to $r$.
By the Caccioppoli inequality in Lemma \ref{CaccLem},
\begin{equation*}
\|\nabla u\|_{L^2(D_2)}\leq
C\pr{\frac{1}{r_0}+K^{\frac{s}{s-n}}+M^{\frac{t}{2t-n}}}\|
u\|_{L^2\pr{B_{2r_0}\backslash B_{{r_0}/{2}}}}
\end{equation*}
and
\begin{equation*}
\|\nabla u\|_{L^2(D_3)}\leq
C\pr{\frac{1}{R_1}+K^{\frac{s}{s-n}}+M^{\frac{t}{2t-n}}}\|
u\|_{L^2\pr{B_{R_1}\backslash B_{{R_1}/{4}}}}.
\end{equation*}
Therefore,
\begin{eqnarray*}
J
&\leq& C\pr{K+|\log r_0|} r_0^{-\frac{n}{2}}e^{-\tau
\phi(r_0)}\pr{1+ r_0K^{\frac{s}{s-n}}+r_0M^{\frac{t}{2t-n}}}\|u\|_{L^2(B_{2r_0})}  \nonumber \medskip \\
&+& C\pr{K+|\log R_1|} {R_1}^{-\frac{n}{2}}e^{-\tau
\phi\pr{\frac{R_1}{2}}}\pr{1+R_1K^{\frac{s}{s-n}}+R_1M^{\frac{t}{2t-n}}}
\|u\|_{L^2(B_{R_1})}.
\end{eqnarray*}
Set $D_4=\{r\in D_1, \ r\leq r_1\}$.
From \eqref{jjj} and that $\tau \ge 1$ and $\be_0 > 0$, we have,
\begin{align*}
\| u\|_{L^2 (D_4)}
&\le \tau^{\beta_0}\| u\|_{L^2 (D_4)}
\le \tau^{\beta_0} \|e^{\tau\phi(r)}(\log r) r^{\frac{n}{2}}\|_{L^\iny\pr{D_2}} \| (\log r)^{-1} e^{-\tau\phi(r)} u\|_{L^2 (D_4, r^{-n}dx)}  \\
&\le e^{\tau \phi(r_1)} |\log r_1| r_1^{\frac{n}{2}} J,
\end{align*}
where we have used that $e^{\tau\phi(r)}(\log r) r^{\frac{n}{2}}$ is increasing on $D_1$ for $R_0$ sufficiently small.
Adding $\|u\|_{L^2 \pr{B_{3r_0/2}}}$ to both sides of the last inequality and using the bound on $J$ from above, we get
\begin{align*}
\| u\|_{L^2 (B_{r_1})}
&\le C |\log r_1|  \pr{K+|\log r_0|} \pr{\frac{r_1}{r_0}}^{\frac{n}{2}}e^{\tau \brac{\phi(r_1)- \phi(r_0)}}\pr{1+ r_0K^{\frac{s}{s-n}}+r_0M^{\frac{t}{2t-n}}}\|u\|_{L^2(B_{2r_0})} \\
&+ C |\log r_1| \pr{K+|\log R_1|}
\pr{\frac{r_1}{R_1}}^{\frac{n}{2}}e^{\tau\brac{\phi(r_1) -
\phi\pr{\frac{R_1}{2}}}}\nonumber \\
&\times\pr{1+R_1K^{\frac{s}{s-n}}+R_1M^{\frac{t}{2t-n}}}
\|u\|_{L^2(B_{R_1})}.
\end{align*}
Let $U_1 =\|u\|_{L^2(B_{2r_0})}$, $U_2=\|u\|_{L^2(B_{R_1})}$ and define
\begin{align*}
A_1 &= C |\log r_1|  \pr{K+|\log r_0|} \pr{\frac{r_1}{r_0}}^{\frac{n}{2}} \pr{1+ r_0K^{\frac{s}{s-n}}+r_0M^{\frac{t}{2t-n}}} \\
A_2 &= C |\log r_1| \pr{K+|\log R_1|} \pr{\frac{r_1}{R_1}}^{\frac{n}{2}} \pr{1+R_1K^{\frac{s}{s-n}}+R_1M^{\frac{t}{2t-n}}}.
\end{align*}
Then the previous inequality simplifies to
\begin{eqnarray}
\| u\|_{L^2 (B_{r_1})}  &\leq& A_1 \brac{\frac{\exp \pr{\phi(r_1)}}{\exp\pr{ \phi(r_0)}}}^\tau U_1 + A_2 \brac{\frac{\exp\pr{ \phi(r_1)}}{\exp\pr{ \phi\pr{\frac{R_1}{2}}}}}^\tau U_2.
\label{D4est}
\end{eqnarray}
Introduce
$$\frac{1}{k_0}=\frac{\phi(\frac{R_1}{2})-\phi(r_0)}{\phi(\frac{R_1}{2})-\phi(r_1)}.$$
Recall that $\phi(r)=\log r+\log (\log r)^2$.
If $r_1$ and $R_1$ are fixed, and $r_0\ll r_1$, i.e. $r_0$ is sufficiently small, then $\frac{1}{k_0}\simeq \log \frac{1}{r_0}$.
Let
$$\tau_1 =\frac{k_0}{\phi\pr{\frac{R_1}{2}}-\phi(r_1)}\log\pr{\frac{A_2{U}_2}{A_1 {U}_1}}.$$
If $\tau_1 \ge C_1 K^\kappa + C_2 M^\mu$, then the above calculations are valid with $\tau = \tau_1$ and by substituting $\tau_1$ into \eqref{D4est}, we get
\begin{align}
\| u\|_{L^2 (B_{r_1})}
&\leq 2\pr{A_1 U_1}^{k_0}\pr{A_2 U_2}^{1 - k_0}.
\label{mix1}
\end{align}
On the other hand, if $\tau_1 < C_1 K^\kappa + C_2 M^\mu$, then
\begin{align*}
U_2
< \frac{A_1}{A_2} \exp\brac{\pr{C_1 K^\kappa + C_2 M^\mu} \pr{\phi\pr{\frac{R_1}{2}}-\phi(r_0)}} U_1.
\end{align*}
The last inequality implies that
\begin{equation}
\|u\|_{L^2 (B_{r_1})}
\le C \pr{\frac{R_1}{r_0}}^{\frac n 2} \pr{1 +\frac{|\log r_0|}{K}} e^{\pr{C_1 K^\kappa + C_2 M^\mu} \pr{\phi\pr{\frac{R_1}{2}}-\phi(r_0)}} \|u\|_{L^2(B_{2r_0})}.
\label{mix2}
\end{equation}
By combining \eqref{mix1} and \eqref{mix2}, we arrive at
\begin{align}
\| u\|_{L^2 (B_{r_1})}
&\le C  |\log r_1| r_1^{\frac n 2}\brac{ r_0^{-\frac{n}{2}}\pr{K+|\log r_0|} \pr{1+ r_0K^{\frac{s}{s-n}}+r_0M^{\frac{t}{2t-n}}} \|u\|_{L^2(B_{2r_0})}}^{k_0} \nonumber \\
&\times \brac{R_1^{-\frac{n}{2}}\pr{K+|\log R_1|} \pr{1+R_1K^{\frac{s}{s-n}}+R_1M^{\frac{t}{2t-n}}} \|u\|_{L^2(B_{R_1})}}^{1 - k_0} \nonumber \\
&+C \pr{\frac{R_1}{r_0}}^{\frac n 2} \pr{1 +\frac{|\log r_0|}{K}}
e^{\pr{C_1 K^\kappa + C_2 M^\mu}
\pr{\phi\pr{\frac{R_1}{2}}-\phi(r_0)}} \|u\|_{L^2(B_{2r_0})}.
 \label{end2}
\end{align}
By elliptic regularity (see for example \cite{HL11}, \cite{GT01}), we see that
\begin{align}
\|u\|_{L^\infty(B_r)}
\le C\pr{1 + r K^{\frac{s}{s-n}} + r M^{\frac{t}{2t-n}}}^{\frac n 2}r^{- \frac n 2}\|u\|_{L^2(B_{2r})}.
\label{ell}
\end{align}
From \eqref{end2} and \eqref{ell}, we get the three-ball inequality in the $L^\infty$-norm that is given in \eqref{three}.
\end{proof}

The inequality \eqref{three} is the three-ball inequality we use in the proof of Theorem \ref{thh}.
We first use \eqref{three} in the propagation of smallness argument to establish a lower bound for the solution in $B_r$.
Similar arguments have been performed in \cite{Zhu16}.
Then we use \eqref{three} again to establish the order of vanishing estimate.

\begin{proof} [Proof of Theorem  \ref{thh}] Without loss of
generality, we may assume that $x_0$ is the origin.  With
$r_0=\frac{r}{2}$, $r_1=4r$ and $R_1=10r$, it follows from
\eqref{three} that
\begin{align}
\|u\|_{L^\infty \pr{B_{3r}}}
&\le C \pr{1 + K^{\frac{s}{s-n}} + M^{\frac{t}{2t-n}}}^{1+\frac n 2} \pr{K+|\log r|} |\log r|  \|u\|_{L^\iny(B_{r})}^{k_0} \|u\|_{L^\iny(B_{10r})}^{1 - k_0} \nonumber \\
&+C\pr{1 + \frac{\abs{\log r}}{K}} \exp\brac{ \pr{C_1 K^\kappa + C_2 M^\mu} \pr{\phi\pr{5r}-\phi\pr{\frac r 2}} } \|u\|_{L^\iny(B_{r})},
\label{refi}
\end{align}
where $\disp k_0 = \frac{\phi(5r)-\phi(4r)}{\phi(5r)-\phi\pr{\frac r 2}}$.
We can check that
$$c^{-1} \leq \phi(5r)-\phi\pr{\frac{r}{2}}\leq c \quad \mbox{and} \quad c^{-1} \leq \phi(5r)-\phi(4r)\leq c,$$
where $c$ is some universal constant.
Therefore, $k_0$ is independent of $r$ in this case.

We choose a small $r < \frac 1 2$ such that
$$\sup_{B_r(0)}|u|=\ell,$$
where $\ell>0$.
Since \eqref{normal} implies $\disp \sup_{|x|\leq 1}|u(x)|\geq 1$, there exists some $\bar x\in B_1$ such that $\disp \abs{u(\bar x)}=\sup_{|x|\leq 1}|u(x)|\geq 1$.
We select a sequence of balls, each with radius $r$, centered at $x_0=0, \ x_1, \ldots, x_d$ so that
$x_{i+1}\in B_{r}(x_i)$ for every $i$, and $\bar x\in B_{r}(x_d)$.
Note that the number of balls, $d$, depends on the radius $r$ which is to be fixed.
Employing the $L^\infty$-version of  three-ball inequality (\ref{refi}) at the origin and the boundedness assumption of $u$ given in \eqref{bound}, we get
\begin{align*}
\|u\|_{L^\infty \pr{B_{3r}(0)}}
&\le C \hat C^{1 - k_0} \pr{1 + K^{\frac{s}{s-n}} + M^{\frac{t}{2t-n}}}^{1+\frac n 2} \ell^{k_0} \pr{1+\frac{|\log r|}{K}} K \abs{\log r}   \nonumber \\
&+ C \ell \pr{1 + \frac{|\log r|}{K}} \exp\brac{c\pr{C_1 K^\kappa + C_2 M^\mu} } \nonumber \\
&\le C \pr{\hat C^{1 - k_0} + 1}\pr{\ell^{k_0} + \ell} \pr{K+ |\log r| } \abs{\log r}\exp\pr{C_1 K^\kappa + C_2 M^\mu},
\end{align*}
where we have used that $K, M \ge 1$ to conclude that $\pr{1 + K^{\frac{s}{s-n}} + M^{\frac{t}{2t-n}}}^{1+\frac n 2} \le \exp\pr{C_1 K^\kappa + C_2 M^\mu}$, after possibly redefining $C_1$ and $C_2$, still depending on $n$, $s$, and $t$.
Since $B_r(x_{i+1})\subset B_{3r}(x_{i})$, then for every $i = 1, 2, \ldots, d$,
\begin{equation}
\|u\|_{L^\infty (B_r(x_{i+1}))}\leq  \|u\|_{L^\infty
(B_{3r}(x_{i}))}. \label{bbb}
\end{equation}
Repeating the above argument with balls centered at $x_i$ and using \eqref{bbb}, we obtain
\begin{equation*}
\|u\|_{L^\infty (B_{3r}(x_{i}))}
\leq C_i \ell^{D_i} \brac{\pr{K+ |\log r| } \abs{\log r}}^{E_i} \exp\brac{F_i\pr{C_1 K^\kappa + C_2 M^\mu} }
\end{equation*}
for $i=0, 1, \cdots, d$, where $C_i$ depends on $n$, $d$, $\hat C$, and $C$ from Lemma \ref{threeBalls}, and $D_i$, $E_i$, $F_i$ are constants depending on $n$ and $d$.
By the fact that $\abs{u(\bar x)} \geq 1$ and $\bar x \in B_{3r}(x_d)$, we obtain
\begin{align*}
\ell &\ge C_d^{-\frac{1}{D_d}} \exp\brac{-\frac{F_d}{D_d}\pr{C_1 K^\kappa + C_2 M^\mu} } \brac{\pr{K+ |\log r| } \abs{\log r}}^{-\frac{E_d}{D_d}} \\
&= c \exp\brac{-C \pr{C_1 K^\kappa + C_2 M^\mu} }\pr{K+ |\log r|}^{-C} |\log r|^{-C}
\end{align*}
where $c$ and $C$ are new constants with $c\pr{n, s, t, d, \hat C}$ and $C\pr{n, d}$.

Now we fix the radius $r$ as a small number so that $d$ is a fixed constant.
We are going to use the three-ball inequality again with a different set of radii.
Let $\frac{3}{4}r_1=r$, $R_1=10r$ and let $r_0 << r$, i.e. $r_0$ is sufficiently small with respect to $r$.
Then, by the three-ball inequality \eqref{three},
$$ \ell \leq {\rm I} +\Pi,$$
where
\begin{align*}
{\rm I} &= C F\pr{r}^{\frac n 2}  |\log r| \brac{ (K+|\log r_0|) F\pr{r_0} \|u\|_{L^\iny(B_{2r_0})}}^{k_0} \brac{ (K+|\log 10r|) F\pr{10 r} \|u\|_{L^\iny(B_{10 r})}}^{1 - k_0} \\
\Pi &= C F\pr{r}^{\frac n 2} \pr{\frac{r }{r_0 }}^{\frac n 2} \pr{1 +\frac{|\log r_0|}{K}} e^{\pr{C_1 K^\kappa + C_2 M^\mu} \pr{\phi\pr{5r}-\phi(r_0)}} \|u\| _{L^\iny(B_{2r_0})},
\end{align*}
with $\disp k_0 = \frac{\phi(5r)-\phi(\frac 4 3 r)}{\phi(5r)-\phi(r_0)}$ and $F\pr{r} = 1 + r K^{\frac{s}{s-n}} + r M^{\frac{t}{2t-n}}$.

On one hand, if ${\rm I} \leq \Pi$, then
\begin{align*}
&c \exp\brac{-C \pr{C_1 K^\kappa + C_2 M^\mu} }\pr{1+\frac{|\log r|}{K}}^{-C} \pr{K |\log r|}^{-C}
\le \ell \le 2 \Pi \\
&\le 2 C F\pr{r}^{\frac n 2} \pr{\frac{r }{r_0 }}^{\frac n 2} \pr{1 +\frac{|\log r_0|}{K}} e^{\pr{C_1 K^\kappa + C_2 M^\mu} \pr{\phi\pr{5r}-\phi(r_0)}} \|u\|_{L^\iny(B_{2r_0})}.
\end{align*}
so that
\begin{align*}
&4 C \pr{ |\log r|^2+|\log r|}^{C} r^{\frac n 2} \|u\|_{L^\iny(B_{2r_0})} \\
&\ge c \exp\set{\brac{\phi(r_0) -\pr{\phi\pr{5r}+C }}\pr{C_1 K^\kappa + C_2 M^\mu} - \frac n 2 \log F\pr{1} - C \log K+ \frac n 2 \log r_0 -  \log\abs{\log r_0}}.
\end{align*}
Keeping in mind that $r$ is a fixed small positive constant, there exists $\hat r_0 << r$ so that for any $r_0 \le \hat r_0$,
$$\phi(r_0) -\pr{\phi\pr{5r}+C } \ge 2 \phi\pr{r_0}.$$
and
$$-\phi\pr{r_0} \pr{C_1 K^\kappa + C_2 M^\mu} \ge \frac n 2 \log F\pr{1} + C \log K + \frac n 2 \abs{\log r_0} + \log\abs{\log r_0},$$
where we are using that $K, M \ge 1$ and possibly adjusting the definitions of $C_1$ and $C_2$.
It follows that
\begin{align*}
&\brac{\phi(r_0) -\pr{\phi\pr{5r}+C }}\pr{C_1 K^\kappa + C_2 M^\mu} - \frac n 2 \log F\pr{1} - C \log K+ \frac n 2 \log r_0 -  \log\abs{\log r_0} \\
&\ge 3 \phi\pr{r_0}\pr{C_1 K^\kappa + C_2 M^\mu}
\end{align*}
and therefore
\begin{align*}
\|u\|_{L^\iny(B_{2r_0})}
&\ge c r_0^{ \pr{C_1 K^\kappa + C_2 M^\mu} },
\end{align*}
where now $c\pr{n, s, t, \hat C}$.
On the other hand, if $\Pi \leq {\rm I}$, then
\begin{align*}
&c \exp\brac{-C \pr{C_1 K^\kappa + C_2 M^\mu} }\pr{1+\frac{|\log r|}{K}}^{-C} |\log r|^{-C}
\le \ell \le 2 {\rm I} \\
&\le 2 C F\pr{r}^{\frac n 2} |\log r|
\brac{(K+|\log r_0|) F\pr{r_0} \|u\|_{L^\iny(B_{2r_0})}}^{k_0}
\brac{(K+|\log 10r|)  F\pr{10 r} \|u\|_{L^\iny(B_{10 r})}}^{1 - k_0} \\
&\le 2 C \hat C K F\pr{1}^{1 + \frac n 2} (|\log r|^2+|\log r|)
\brac{\frac{(1+|\log r_0|) \|u\|_{L^\iny(B_{2r_0})}}{(1+|\log r|) \hat C}}^{k_0} ,
\end{align*}
where we have used that $\|u\|_{L^\infty\pr{B_{10r}}}\leq \hat{C}$ from \eqref{bound}.
Then, raising both sides to $\frac{1}{k_0}$, we obtain
\begin{align*}
\|u\|_{L^\iny(B_{2r_0})} &\ge\frac{\hat C |\log r| }{2}
\exp\Big[-\frac{1}{k_0} \pr{CC_1 K^\kappa + CC_2 M^\mu  + \log K +
\frac{n+2}{2 } \log F\pr{1} + \log A} \\ & - \log |\log r_0| \Big],
\end{align*}
where $A := \frac{2 C \hat C (|\log r|^2+|\log r|)^{1+C} }{c} \ge 1$.
Recalling that $\disp -\frac{1}{k_0} = \frac{\phi(r_0) - \phi(5r)}{\phi(5r)-\phi(\frac 4 3 r)} \ge \frac{2\phi(r_0) + C}{\phi(5r)-\phi(\frac 4 3 r)}$ by our previous assumption that $r_0 \le \hat r_0$,  we see that $\disp -\frac{1}{k_0} \sim \phi\pr{r_0} = \log r_0 + \log\pr{\log r_0}^2$.
After possibly redefining $C_1$ and $C_2$ again, we have
\begin{align*}
 \|u\|_{L^\iny(B_{2r_0})}
 &\ge c r_0^{\pr{C_1 K^\kappa + C_2 M^\mu}},
\end{align*}
as before.
This completes the proof of Theorem \ref{thh}.
\end{proof}

To prove Theorem \ref{thhh}, we require another three-ball inequality.

\begin{lemma}
Let $0 < r_0< r_1< R_1 < R_0$, where $R_0 < 1$ is sufficiently small.
Let $t \in \pb{ \frac{4n^2}{7n+2}, \iny}$.
Assume that for some $M \ge 1$, $\|V\|_{L^t\pr{B_{R_0}}} \le M$.
Let $u : B_{R_{0}} \to \C$ be a solution to \eqref{goal1} in $B_{R_0}$.
Then there exists a constant $C$, depending on $n$ and $t$, such that
\begin{align}
\|u\|_{L^\infty \pr{B_{3r_1/4}}}
&\le C |\log r_1| F\pr{r_1}^{\frac n 2} \pr{|\log r_0| F\pr{r_0} \|u\|_{L^\iny(B_{2r_0})}}^{k_0} \pr{ |\log R_1|F\pr{R_1} \|u\|_{L^\iny(B_{R_1})}}^{1 - k_0} \nonumber  \\
&+C F\pr{r_1}^{\frac n 2}  \pr{\frac{R_1}{r_1}}^{\frac n
2}\frac{\abs{\log r_0}}{\abs{\log R_1}} e^{C_2 M^\mu \pr{\phi\pr{\frac{R_1}{2}}-\phi(r_0)}} \|u\|_{L^\iny(B_{2r_0})},
\label{three2}
\end{align}
where $\disp k_0 = \frac{\phi(\frac{R_1}{2})-\phi(r_1)}{\phi(\frac{R_1}{2})-\phi(r_0)}$, $F\pr{r} = 1 + r M^{\frac{t}{2t-n}}$, and $\mu$ and $C_2$ are as in Theorem \ref{CarlpqV}.
\end{lemma}

\begin{proof}
Let $r_0< r_1< R_1$.
Choose a smooth function $\eta\in C^\infty_{0}(B_{R_0})$ with $B_{2R_1}\subset B_{R_0}$.
As before, let
$$D_1=\brac{\frac{3}{2}r_0, \frac{1}{2}R_1 }, \quad  \quad
D_2= \brac{r_0, \frac{3}{2}r_0}, \quad \quad
D_3=\brac{\frac{1}{2}R_1, \frac{3 }{4}R_1}.$$
Let $\eta=1$ on $D_1$ and $\eta=0$ on $[0, \ r_0]\cup \brac{\frac{3}{4}R_1, \ R_1}$.
Then $|\nabla \eta|\leq \frac{C}{r_0}$ and $|\nabla^2\eta|\leq \frac{C}{r_0^2}$ on $D_2$, and $|\nabla \eta|\leq \frac{C}{R_1}$ and $|\nabla^2 \eta|\leq\frac{C}{R_1^2}$ on $D_3$.

Since $u$ is a solution to \eqref{goal1} in $B_{R_0}$, then $u \in L^\iny\pr{B_{R_1}} \cap W^{1,2}\pr{B_{R_1}} \cap W^{2,p}\pr{B_{R_1}}$, so by regularization, the estimate in Theorem \ref{CarlpqV} holds for $\eta u$.
Using that $u$ is a solution to equation \eqref{goal1}, we get
\begin{align*}
\tau^{\be_0} \|(\log r)^{-1} e^{-\tau \phi(r)} \eta u\|_{L^2(r^{-n}dx)}
&\leq  C_0 \|(\log r ) e^{-\tau \phi(r)} r^2\pr{ \LP \eta \, u + 2 \gr \eta \cdot \gr u  }\|_{L^p(r^{-n} dx)},
\end{align*}
whenever
$$\tau \ge C_2 M^{\mu}.$$

Then
\begin{equation*}
\tau^{\beta_0} \|(\log r)^{-1} e^{-\tau \phi(r)} u\|_{L^2(D_1, r^{-n}dx )}\leq J,
\end{equation*}
where
$$ J= C_0 \|(\log r) e^{-\tau \phi(r)} r^2 (\LP \eta \, u + 2\nabla \eta \cdot \nabla u)\|_{L^p(D_2\cup D_3, r^{-n} dx)}.$$
Set $D_4=\{r\in D_1, \ r\leq r_1\}$.
From the previous inequality and that $\tau \ge 1$ and $\be_0 > 0$, we have, as before, that
\begin{align}
\| u\|_{L^2 (D_4)}
&\le e^{\tau \phi(r_1)} |\log r_1| r_1^{\frac{n}{2}} J.
\label{JBound}
\end{align}

As in the previous proof,
\begin{eqnarray*}
J &\leq & C |\log r_0| r_0^{-\frac{n}{2}} e^{-\tau \phi(r_0)}  \pr{\| u\|_{{L^2(D_2)}}+ r_0
\|\nabla u\|_{{L^2(D_2)}}} \nonumber \medskip \\
&+& C|\log R_1| R_1^{-\frac{n}{2}}e^{-\tau \phi\pr{\frac{R_1}{2}}} \pr{\|
u\|_{{L^2(D_3)}}+ R_1 \| \nabla u\|_{{L^2(D_3)}}}.
\end{eqnarray*}
By the Caccioppoli inequality in Lemma \ref{CaccLem},
\begin{equation*}
\|\nabla u\|_{L^2(D_2)}\leq
C\pr{\frac{1}{r_0}+M^{\frac{t}{2t-n}}}\|
u\|_{L^2\pr{B_{2r_0}\backslash B_{{r_0}/{2}}}}
\end{equation*}
and
\begin{equation*}
\|\nabla u\|_{L^2(D_3)}\leq
C\pr{\frac{1}{R_1}+M^{\frac{t}{2t-n}}}\|
u\|_{L^2\pr{B_{R_1}\backslash B_{{R_1}/{4}}}}.
\end{equation*}
Therefore,
\begin{eqnarray*}
J
&\leq& C |\log r_0| r_0^{-\frac{n}{2}}e^{-\tau \phi(r_0)}\pr{1+r_0M^{\frac{t}{2t-n}}}\|u\|_{L^2(B_{2r_0})}  \nonumber \medskip \\
&+& C |\log R_1| {R_1}^{-\frac{n}{2}}e^{-\tau \phi\pr{\frac{R_1}{2}}}\pr{1+R_1M^{\frac{t}{2t-n}}}
\|u\|_{L^2(B_{R_1})}.
\end{eqnarray*}
Adding $\|u\|_{L^2 \pr{B_{3r_0/2}}}$ to both sides of \eqref{JBound}, we get
\begin{align*}
\| u\|_{L^2 (B_{r_1})}
&\le C |\log r_1|  |\log r_0| \pr{\frac{r_1}{r_0}}^{\frac{n}{2}}e^{\tau \brac{\phi(r_1)- \phi(r_0)}}\pr{1+ r_0M^{\frac{t}{2t-n}}}\|u\|_{L^2(B_{2r_0})} \\
&+ C |\log r_1| |\log R_1| \pr{\frac{r_1}{R_1}}^{\frac{n}{2}}e^{\tau\brac{\phi(r_1) -
\phi\pr{\frac{R_1}{2}}}}\pr{1+R_1M^{\frac{t}{2t-n}}} \|u\|_{L^2(B_{R_1})}.
\end{align*}
Let $U_1 =\|u\|_{L^2(B_{2r_0})}$, $U_2=\|u\|_{L^2(B_{R_1})}$, and this time define
\begin{align*}
A_1 &= C |\log r_1|  |\log r_0| \pr{\frac{r_1}{r_0}}^{\frac{n}{2}} \pr{1+r_0M^{\frac{t}{2t-n}}} \\
A_2 &= C |\log r_1| |\log R_1| \pr{\frac{r_1}{R_1}}^{\frac{n}{2}} \pr{1+R_1M^{\frac{t}{2t-n}}}.
\end{align*}
Then the previous inequality simplifies to \eqref{D4est}.
With $k_0$ as before, i.e. $k_0=\frac{\phi(\frac{R_1}{2})-\phi(r_1)}{\phi(\frac{R_1}{2})-\phi(r_0)}$,
let
$$\tau_1 =\frac{k_0}{\phi\pr{\frac{R_1}{2}}-\phi(r_1)}\log\pr{\frac{A_2{U}_2}{A_1 {U}_1}}.$$
If $\tau_1 \ge C_2 M^\mu$, then the above calculations are valid with $\tau = \tau_1$ and by substituting $\tau_1$ into \eqref{D4est}, we get \eqref{mix1}.
On the other hand, if $\tau_1 < 1 + C_2 M^\mu$, then
\begin{align*}
U_2
< \frac{A_1}{A_2} \exp\brac{ C_2 M^\mu \pr{\phi\pr{\frac{R_1}{2}}-\phi(r_0)}} U_1.
\end{align*}
The last inequality implies that
\begin{equation}
\|u\|_{L^2 (B_{r_1})}
\le C \pr{\frac{R_1}{r_0}}^{\frac n 2}\frac{\abs{\log r_0}}{\abs{\log R_1}} e^{ C_2 M^\mu  \pr{\phi\pr{\frac{R_1}{2}}-\phi(r_0)}} \|u\|_{L^2(B_{2r_0})}.
\label{mix22}
\end{equation}
By combining \eqref{mix1} and \eqref{mix22}, we arrive at
\begin{align}
\| u\|_{L^2 (B_{r_1})}
&\le C r_1^{\frac n 2} |\log r_1|  \brac{ r_0^{- \frac n 2} |\log r_0| \pr{1+r_0M^{\frac{t}{2t-n}}} \|u\|_{L^2(B_{2r_0})}}^{k_0} \nonumber \\
&\times\brac{ R_1^{- \frac n 2}|\log R_1| \pr{1+R_1M^{\frac{t}{2t-n}}} \|u\|_{L^2(B_{R_1})}}^{1 - k_0} \nonumber  \\
&+C \pr{\frac{R_1}{r_0}}^{\frac n 2}\frac{\abs{\log r_0}}{\abs{\log R_1}} e^{ C_2 M^\mu \pr{\phi\pr{\frac{R_1}{2}}-\phi(r_0)}} \|u\|_{L^2(B_{2r_0})}.
 \label{end22}
\end{align}
By elliptic regularity
\begin{align}
\|u\|_{L^\infty(B_r)}
\le C\pr{1 + r M^{\frac{t}{2t-n}}}^{\frac n 2}r^{- \frac n 2}\|u\|_{L^2(B_{2r})}.
\label{ell2}
\end{align}
From \eqref{end22} and \eqref{ell2}, we get the three-ball inequality in the $L^\infty$-norm that is given in \eqref{three2}.
\end{proof}

The proof of Theorem \ref{thhh} follows the arguments in the proof of Theorem \ref{thh}, except that we use \eqref{three2} in place of \eqref{three}.

\section{Unique continuation at infinity}
\label{QuantUC}

Using the scaling arguments established in \cite{BK05}, we show how the unique continuation estimates at infinity follow from the maximal order of vanishing estimates.

\begin{proof}[Proof of Theorem \ref{UCVW}]
Let $u$ be a solution to \eqref{goal} in $\R^n$.
Fix $x_0 \in \R^n$ and set $\abs{x_0} = R$.
Let $u_R(x) = u(x_0 + Rx)$.
Define $W_R\pr{x} = R \, W\pr{x_0 + R x}$ and $V_R\pr{x} = R^2 V\pr{x_0 + R x}$.
For any $r > 0$,
\begin{align*}
\norm{W_R}_{L^s\pr{B_r\pr{0}}}
&= \pr{\int_{B_r\pr{0}} \abs{W_R\pr{x}}^s dx}^{\frac 1 s}
= \pr{\int_{B_r\pr{0}} \abs{R \, W\pr{x_0 + R x}}^s dx}^{\frac 1 s} \\
&= R^{1 - \frac n s } \pr{  \int_{B_r\pr{0}} \abs{W\pr{x_0 + R x}}^s d\pr{Rx} }^{\frac 1 s}
= R^{1 - \frac n s } \norm{W}_{L^s\pr{B_{r R}\pr{x_0}}}.
\end{align*}
and
\begin{align*}
\norm{V_R}_{L^t\pr{B_r\pr{0}}}
&= R^{2 - \frac n t } \pr{  \int_{B_r\pr{0}} \abs{V\pr{x_0 + R x}}^t d\pr{Rx} }^{\frac 1 t}
= R^{2 - \frac n t } \norm{V}_{L^t\pr{B_{r R}\pr{x_0}}}.
\end{align*}
Therefore, $\disp \norm{W_R}_{L^s\pr{B_{10}\pr{0}}} = R^{1 - \frac n s} \norm{W}_{L^s\pr{B_{10R}\pr{x_0}}} \le A_1 R^{1 - \frac n s}$ and $\disp \norm{V_R}_{L^t\pr{B_{10}\pr{0}}} \le A_0 R^{2 - \frac n t}$.
Moreover,
\begin{align*}
& \LP u_R\pr{x} + W_R\pr{x} \cdot \gr u_R\pr{x}  + V_R\pr{x} u_R\pr{x} \\
&= R^2 \LP u\pr{x_0 + R x} + R^2 W\pr{x_0 + Rx} \cdot \gr u\pr{x_0 + R x}  + R^2 V\pr{x_0 + R x}u\pr{x_0 + R x}
= 0.
\end{align*}
Therefore, $u_R$ is a solution to a scaled version of \eqref{goal} in $B_{10}$.
Clearly,
\begin{align*}
\norm{u_R}_{L^\iny\pr{B_{10}}}
&= \norm{u}_{L^\iny\pr{B_{10R}\pr{x_0}}} \le C_0.
\end{align*}
Note that for $\disp\widetilde{x_0} := -x_0/R$, $\disp| \widetilde{x_0}| = 1$ and $\abs{u_R(\widetilde{x_0})} = \abs{u(0)} \ge 1$ so that $\disp\norm{u_R}_{L^\iny(B_1)} \ge 1$.
Thus, if $R$ is sufficiently large, then we may apply Theorem \ref{thh} to $u_R$ with $K = A_1 R^{1 - \frac n s}$, $M = A_2 R^{2 - \frac n t}$, and $\hat C = C_0$ to get
\begin{align*}
\norm{u}_{L^\iny\pr{{B_{1}(x_0)}}} = & \norm{u_R}_{L^\iny\pr{B_{1/R}(0)}}  \\
\ge & c(1/R)^{^{\brac{C_1 \pr{A_1 R^{1 - \frac n s}}^\kappa + C_2 \pr{A_2 R^{2 - \frac n t} }^\mu}}} \\
= & \exp\set{-\brac{C_1 \pr{A_1 R^{1 - \frac n s}}^\kappa + C_2 \pr{A_2 R^{2 - \frac n t} }^\mu} \log R - \log c}.
\end{align*}
Since
\begin{align*}
\max\set{\kappa\pr{1 - \frac n s}, \mu\pr{2 - \frac n t} }= \Pi := \left\{\begin{array}{ll}
\frac{4\pr{s-n}}{2s - \pr{3n-2}} & t > \frac{sn}{s+n}, \medskip \\
\frac{4 \pr{t-n\frac t s}}{\pr{5 - \frac 2 n}s - \pr{3n-2} \frac s
t} & n\pr{\frac{3n-2}{5n-2}} < t \le \frac{sn}{s+n},
\end{array}\right.
\end{align*}
then
\begin{align*}
\norm{u}_{L^\iny\pr{{B_{1}(x_0)}}}
\ge & \exp\brac{-\pr{C_1 A_1^\kappa + C_2 A_2^\mu + \log c} R^\Pi \log R}
\end{align*}
and the conclusion of the theorem follows.
\end{proof}

Corollary \ref{UCW} follows from Theorem \ref{UCVW} with $V \equiv 0$.
Finally, we present the proof of Theorem \ref{UCV}.

\begin{proof}[Proof of Theorem \ref{UCV}]
Let $u$ be a solution to \eqref{goal1} in $\R^n$.
Fix $x_0 \in \R^n$ and set $\abs{x_0} = R$.
As above, let $u_R(x) = u(x_0 + Rx)$ and $V_R\pr{x} = R^2 V\pr{x_0 + R x}$.
Then, $\disp \norm{V_R}_{L^t\pr{B_{10}\pr{0}}} \le A_0 R^{2 - \frac n t}$ and
\begin{align*}
 \LP u_R\pr{x} + V_R\pr{x} u_R\pr{x}
&= R^2 \LP u\pr{x_0 + R x} + R^2 V\pr{x_0 + R x}u\pr{x_0 + R x}
= 0.
\end{align*}
Therefore, $u_R$ is a solution to a scaled version of \eqref{goal1} in $B_{10}$.
Clearly,
\begin{align*}
\norm{u_R}_{L^\iny\pr{B_{10}}}
&= \norm{u}_{L^\iny\pr{B_{10R}\pr{x_0}}} \le C_0.
\end{align*}
Note that for $\disp\widetilde{x_0} := -x_0/R$, $\disp| \widetilde{x_0}| = 1$ and $\abs{u_R(\widetilde{x_0})} = \abs{u(0)} \ge 1$ so that $\disp\norm{u_R}_{L^\iny(B_1)} \ge 1$.
Thus, if $R$ is sufficiently large, then we may apply Theorem \ref{thhh} to $u_R$ with $M = A_2 R^{2 - \frac n t}$, and $\hat C = C_0$ to get
\begin{align*}
\norm{u}_{L^\iny\pr{{B_{1}(x_0)}}}
&= \norm{u_R}_{L^\iny\pr{B_{1/R}(0)}}
\ge c(1/R)^{^{C_2 \pr{A_2 R^{2 - \frac n t} }^\mu}} \\
&= \exp\set{-\brac{C_2 \pr{A_2 R^{2 - \frac n t} }^\mu} \log R - \log c}.
\end{align*}
Since
\begin{align*}
\mu\pr{2 - \frac n t} = \Pi :=  \left\{\begin{array}{ll}
\frac{4\pr{2t-n}}{6t - \pr{3n-2}} & t > n, \medskip \\
\frac{4\pr{2t-n}} {7t+\frac{2t}n-4n-\eps} & \frac{4n^2}{7n+2} < t
\le n,
\end{array}\right.
\end{align*}
then
\begin{align*}
\norm{u}_{L^\iny\pr{{B_{1}(x_0)}}}
\ge & \exp\brac{-\pr{C_2 A_2^\mu - \log c} R^\Pi \log R},
\end{align*}
as required.
\end{proof}

\noindent
{\small {\bf Conflict of Interest}
 The authors declares that they have no conflict of interest.}


\def\cprime{$'$}

\end{document}